\documentclass[a4paper,DIV=classic,fontsize=10pt]{myart}
\newcommand{\Set}[1]{\ensuremath{ \left\{ #1 \right\} }}
\newcommand{\set}[1]{\ensuremath{ \{ #1 \} }}
\newcommand{\R}{\mathbb{R}}

\newcommand*{\epi}{{\rm epi}\,}
\newcommand*{\hypo}{{\rm hypo}\,}
\newcommand*{\gr}{{\rm gr}\,}
\newcommand*{\cl}{{\rm cl}\,}
\newcommand*{\co}{{\rm co}\,}
\newcommand*{\inte}{{\rm int}\,}

\begin{document}

\title{Complete Duality for Quasiconvex and Convex Set-Valued Functions}

\author[a,1]{Samuel Drapeau}
\author[b,2]{Andreas H. Hamel}
\author[c,3]{Michael Kupper}

\address[a]{Shanghai Jiao Tong University, SAIF/CAFR and Mathematics Department, 211 West Huaihai Road, Shanghai, P.R. China 200030}
\address[b]{Free University of Bolzano, Piazzetta dell'Universit\`a 1, 39031 Brunico, Italy}
\address[b]{Konstanz University, Universit\"atsstra\ss e 10, 78464, Germany}

\eMail[1]{sdrapeau@saif.sjtu.edu.cn}
\eMail[2]{andreas.hamel@unibz.it}
\eMail[3]{kupper@uni-konstanz.de}


\date{\today}
\abstract{
  This paper provides a unique dual representation of set-valued lower semi-continuous quasiconvex and convex functions.
  The results are based on a duality result for increasing set-valued functions.
}
\keyWords{Set-Valued Functions, Quasiconvexity, Dual Representation, Increasing Functions, Fenchel-Moreau}
\keyAMSClassification{49N15, 46A20}
\maketitle
\section*{Introduction}
\addcontentsline{toc}{section}{Introduction}
\markboth{\uppercase{Introduction}}{\uppercase{Introduction}}

Motivated by vector optimization problems, \citet{hamel09, Loehne11Book, HamelLoehne13JOTA-OF} among others, recently developed a duality theory for set-valued functions.
One of their key results is a Fenchel-Moreau type representation for convex functions with closed convex images.
However, in contrast to the scalar case where the convex conjugate is unique within the class of convex proper lower semi-continuous functions, they do not provide uniqueness for the dual representation.
On the other hand, not only for mathematical reasons, \citet{marinacci03} stressed that the uniqueness of the dual representation is of great importance in decision theory.
This applies in particular to the interpretation of preferences such as the uncertainty averse preferences as discussed in \citep{marinacci03}.

A duality, which uniquely identifies primal and dual functions within some given classes, is called complete.
For instance, the Fenchel-Moreau theorem states a complete duality between lower semi-continuous convex and proper functions on a locally convex topological vector space and its dual space by means of the Fenchel-Legendre convex conjugate, which actually is an automorphism.
It has been shown by \citet{drapeau01} that lower semi-continuous, quasiconvex, real-valued functions are in complete duality with the class of so-called maximal risk functions.
This result is based on a complete duality between increasing functions providing a one-to-one relation between the increasing closed convex lower level sets and the class of maximal risk functions.
For the existence of dual representation of quasiconvex functions we refer to \citet{penot01,penot02} and for a similar complete duality in the evenly quasiconvex case to \citet{marinacci04} and the references therein.

The focus of this paper is to study complete duality results for convex and quasiconvex set-valued functions.
To derive the uniqueness in the robust representation of a quasiconvex function its images have to be closed and monotone under some specific lattice operations.
In case of convex functions, the images are even required to be topologically closed and the order cone of the image set to have non-empty interior.
Our first main result, Theorem \ref{thm:onetoone}, states a complete duality between increasing set-valued functions.
In contrast to the scalar case \citep{drapeau01}, such complete duality is more difficult due to the fact that, unlike the canonical ordering in $\mathbb{R}$, the preorders we consider are no longer total.
Based on this first complete duality, our second main result, Theorem \ref{thm:robrep}, shows the complete duality for lower level-closed quasiconvex set-valued functions on a locally convex topological vector space $X$ in terms of the unique representation
\begin{equation*}
	f(x):=\sup_{x^\ast}R(x^\ast,\langle x^\ast,x\rangle)
\end{equation*}
where $R$ is a maximal risk function and the supremum is taken over the dual space $X^\ast$.
In the convex case, the risk function $R$ has additional structure which leads to a uniquely specified Fenchel-Moreau type representation in our third main result, Theorem \ref{thm:robrepconv}.
The results strongly rely on our complete duality results for increasing left- and right-continuous set-valued functions.

Applications of the theory developed in the present work may be found in several areas.
\citet{Shephard70Book} presented a theoretical framework for cost and production functions and defined ``convex structures'' which are nothing else than quasi-convex set-valued functions.
In \cite[Chapter 11]{Shephard70Book}, a duality theory for support functions of such quasi-convex functions on finite-dimensional spaces is developed.
Here we extend this approach in a general set-valued framework.
Further, \citet{JouiniMeddebTouzi04FS, HamelHeyde10SIFIN, HamelHeydeRudloff11MAFE} developed a theory of convex set-valued risk measures which are appropriate instruments for risk evaluation in markets with transaction costs.
Thus, a synthesis of \citep{drapeau01} and the results of the present work is highly desirable and would extend the convex case treated in the above references to the more general quasi-convex one and yield a variety of new set-valued risk assessment indices.
Finally, in vector optimization and multi-criteria decision making, so called cone-quasiconvex functions were introduced by \citet{DinhTheLuc87JMAA}, studied, for example, by \citet{BenoistBorweinPopovici03PAMS} and extended to set-valued functions in \citet{BenoistPopovici03MMOR}.

The paper is organised as follows.
After fixing some preliminary notations, we address in Section \ref{sec:01} the complete lattices of monotone sets, introduce the algebraical notion of monotone closure of sets and their complete lattice property, and finally the definition and properties of increasing, convex or quasiconvex functions.
Section \ref{sec:02} is dedicated to the duality between left and right-continuous increasing set-valued functions.
Our main result -- a complete duality for lower level-closed quasiconvex functions -- is the subject of Section \ref{sec:03}.
In the preparation of the proof, the notion of maximal risk and minimal penalty function and their one-to-one relation is provided.
Section \ref{sec:04} addresses the particular case of convex functions: departing from the existence result in \citep{hamel09}, we provide the characterisation of the uniqueness of this Fenchel-Moreau type representation.

\subsection*{Notation}

Let $Z$ be a vector space and $\mathcal{P}(Z)$ its power set, that is, the set of all subsets including the empty set $\emptyset$.
A set $K\subseteq Z$ is a cone if $\lambda z \in K$ for every $\lambda>0$ and $z \in K$.
The Minkowski sum or difference for subsets of $Z$ is defined as $A+B:=\set{x+y\colon x\in A\text{ and }y \in B}$ and $A-B:=\set{x-y\colon  x \in A\text{ and }y \in B}$ for $A,B \in \mathcal{P}(Z)$.
In particular $A\pm \emptyset=\emptyset \pm A=\emptyset$ for all $A \in \mathcal{P}(Z)$.
A \emph{preordered vector space} $Z:=(Z,\leq)$ is a vector space together with a vector preorder.\footnote{A transitive and reflexive binary relation $\leq$ such that $x\leq y$ is equivalent to $\lambda x+z\leq \lambda y+z$ for every scalar $\lambda >0$ and $z \in Z$.}
Such a vector preorder corresponds uniquely to the convex cone $K:=\set{k\in Z\colon  k\geq 0}$ by means of $x\leq y$ if and only if $y-x \in K$.
We further set $\hat{K}:=K\setminus (-K)$ which is also a convex cone,\footnote{Note that if $\tilde{K}$ is a convex cone for which holds $\tilde{K}\setminus(-\tilde{K})=\tilde{K}$, it follows that the preorder given by $K=\tilde{K}\cup\set{0}$ is such that $\hat{K}=\tilde{K}$.} and denote by $x < y$ if and only if $y-x\in \hat{K}$.
Here, the relation $ < $ determined by $\hat{K}$ corresponds to the asymmetric part of $\leq$.\footnote{That is $x < y$ if and only if $x\leq y$ and it does not hold $y\leq x$.}
We say that a preordered vector space $Z$ is \emph{properly preordered} if, 
\begin{itemize}
  \item there exists $k \in Z$ such that $k > 0$, that is, $\hat{K}\neq \emptyset$ and, 
  \item for every $ k_1,k_2 > 0$, there exists $k_3 > 0$ such that $ k_1,k_2 > k_3$.
\end{itemize}
Note that the assumption $\hat{K}\neq \emptyset$ means that $K$ is not a vector subspace.
If $\hat{K}\neq \emptyset$, then $\set{z \in Z\colon  z^1 < z}$ and $\set{z \in Z\colon  z < z^2}$ are non-empty for every $z_1,z_2 \in Z$, and $\set{z \in Z\colon  z^1< z <z^2}$ is also non-empty as soon as $z^1<z^2$.

\section{Set-Valued Functions}\label{sec:01}

\subsection{Complete Lattices of Monotone Sets}
Let $Z\colon =(Z,\leq)$ be a preordered vector space with corresponding convex cone $K$.
The preorder $\leq$ can be extended from $Z$ to $\mathcal{P}(Z)$ in at least two canonical ways, see for instance \citet{hamel09} and the references therein.
On the one hand,
\begin{equation*}
  A \leqslant B \quad \text{ if, and only if, } \quad  A + K \supseteq B
\end{equation*}
and on the other hand
\begin{equation*}
  A \eqslantless B \quad \text{ if, and only if, } \quad  A \subseteq B-K
\end{equation*}
for $A, B \in \mathcal{P}(Z)$.\footnote{In particular, $\set{z_1} \leqslant \set{z_2}$ or $\set{z_1}\eqslantless \set{z_2}$ if, and only if, $z_1 \leq z_2$. If $K=\set{0}$, then $\set{z_1} \leqslant \set{z_2}$ or $\set{z_1} \eqslantless \set{z_2}$ if, and only if, $z_1 = z_2$.}
Let 
\begin{equation*}
  \mathcal{P}(Z,K):=\Set{A \in \mathcal{P}(Z)\colon A+K=A}
\end{equation*}
be the set of monotone subsets of $Z$. 
The restriction of $\leqslant$ from ${\cal P}(Z)$ to $\mathcal{P}(Z,K)$ coincides with the partial order $\supseteq$, by means of which $(\mathcal{P}(Z,K), \leqslant )$ is an order complete lattice with
\begin{equation*}
  \inf A_i = \bigcup A_i  \quad \text{and} \quad \sup A_i= \bigcap A_i
\end{equation*}
for every family $\left(A_i\right)\subseteq (\mathcal{P}(Z, K),\leqslant)$.
Symmetrically, the restriction of $\eqslantless$ from ${\cal P}(Z)$ to $\mathcal{P}(Z,-K)$ coincides with the partial order $\subseteq$, and $(\mathcal{P}(Z,-K),\eqslantless)$ is also an order complete lattice with
\begin{equation*}
  \inf B_i = \bigcap B_i \quad \text{and} \quad \sup B_i= \bigcup B_i 
\end{equation*}
for each family $\left(B_i\right)\subseteq (\mathcal{P}(Z, -K),\eqslantless)$.
To simplify the notation, we write $\mathcal{P}(Z,K)$ for $\left(\mathcal{P}(Z,K),\leqslant\right)$ and $\mathcal{P}(Z,-K)$ for $(\mathcal{P}(Z,-K),\eqslantless)$.
With these conventions, it holds $\mathcal{P}(Z,\set{0})=(\mathcal{P}(Z),\supseteq)$ and $\mathcal{P}(Z,-\set{0})=(\mathcal{P}(Z),\subseteq)$.
\begin{remark}\label{rem:blabla01}
    For $z \in Z$, $A \in \mathcal{P}(Z,K)$ and $B\in \mathcal{P}(Z,-K)$ it holds
    \begin{gather*}
        A\leqslant \set{z}\quad\text{is equivalent to}\quad z \in A;\\
        \set{z}\eqslantless B\quad\text{is equivalent to}\quad z \in B.
    \end{gather*}
    Due to this fact, for ease of notations, we convene that 
    \begin{gather*}
        A\leqslant z\quad\text{is equivalent to}\quad z \in A;\\
        z\eqslantless B\quad\text{is equivalent to}\quad z \in B.
    \end{gather*}
    Also, due to its intuitive analogy with the total order in $\mathbb{R}$, for $A \in \mathcal{P}(Z,K)$ and $B \in \mathcal{P}(Z,-K)$ we use the notation
    \begin{align*}
        A < z\quad&\text{for}\quad A\leqslant \tilde{z} \text{ for some }\tilde{z}<z;\\
        z < B\quad&\text{for}\quad \tilde{z} \eqslantless B\text{ for some }z< \tilde{z}.
    \end{align*}
    Note however that this strict inequality does not correspond to the asymmetric part of the preorders $\leqslant$ or $\eqslantless$ respectively.
\end{remark}
For $\lambda>0$ and $A$ in ${\cal P}(Z,K)$ or ${\cal P}(Z,-K)$ we use the element-wise multiplication $\lambda A:=\set{\lambda x\colon x\in A}$ extended by $0 A=K$ for $A \in \mathcal{P}(Z,K)$ and $0 B=-K$ for $B \in \mathcal{P}(Z,-K)$.
In particular, $0 \emptyset = K$ in $\mathcal{P}(Z,K)$ and $0 \emptyset =-K$ in $\mathcal{P}(Z,-K)$.
The sets  $\mathcal{P}(Z,K)$ and $\mathcal{P}(Z,-K)$ with the Minkowsky sum and this multiplication by positive scalars are so called \emph{ordered conlinear spaces} in the sense of \citet{hamel09}.

\subsection{Complete Lattices of Monotonically Closed Sets}\label{subsec:bullet}
In this subsection, $Z$ is a preordered vector space with corresponding cone $K=\set{k\in Z\colon  k\geq 0}$ and such that $\hat{K}=\set{k\in Z\colon  k>0}\neq \emptyset$.
For a set $A \in \mathcal{P}(Z)$, we define
\begin{equation*}
  A^\bullet :=\Set{z \in Z\colon  z+\hat{K}\subseteq A}.
\end{equation*}
Since $K+\hat{K}\subseteq \hat{K}$, it follows $A^\bullet+K+\hat{K}\subseteq A^\bullet+\hat K\subseteq A$ and therefore $A^\bullet + K\subseteq A^\bullet$, that is, $A^\bullet\in {\cal P}(Z,K)$.
Furthermore, let
\begin{equation*}
  \mathcal{K}(Z, K):=\Set{A \in \mathcal{P}(Z,K)\colon  A^\bullet=A}.
\end{equation*}
The restriction of $\leqslant$ from $\mathcal{P}(Z,K)$ to $\mathcal{K}(Z,K)$ is also a partial order.
The same holds for $\mathcal{K}(Z,-K)$ with $\eqslantless$.\footnote{Clearly, if $\hat{K}\neq\emptyset$ then
$(-K)\setminus K = -\hat{K} \neq\emptyset$ and the $^\bullet$-operator is given 
by $A^\bullet=\set{z \in Z\colon z-\hat{K}\subseteq A}$.}
\begin{example}
Let $Z=\R^2$ and $A=]x,+\infty[\times ]y,+\infty[$ for $x,y \in \R$.
  \begin{enumerate}
    \item If $\hat{K}=\set{(x,y)\in \R^2\colon x\ge 0\text{ and }y\ge 0}\setminus \set{0}$, then $A^\bullet=A$.
    \item If $\hat{K}=\set{(x,y)\in \R^2\colon x>0\text{ and }y> 0}$, then $A^\bullet=[x,+\infty[\times [y,+\infty[$.
  \end{enumerate}
\end{example}

\begin{proposition}\label{lem:passage}
  For $A,B\in\mathcal{P}(Z,K)$ with $A \leqslant B$, and each family $\left( A_i \right)\subseteq \mathcal{P}(Z,K)$, it holds $A^\bullet \leqslant A$, $(A^\bullet)^\bullet =A^\bullet$, $A^\bullet \leqslant B^\bullet$, and $\sup A_i^\bullet=(\sup A_i)^\bullet$.

  Analogously, for $A, B \in \mathcal{P}(Z,-K)$ with $A\eqslantless B$ and each family $\left( A_i \right)\subseteq \mathcal{P}(Z,-K)$, it holds $A \eqslantless A^\bullet$, $(A^\bullet)^\bullet =A^\bullet$, $A^\bullet \eqslantless B^\bullet$ and $\inf A_i^\bullet=(\inf A_i)^\bullet$.\footnote{Here $\sup$ and $\inf$ are understood in the sense of $(\mathcal{P}(Z,K),\leqslant)$ and $(\mathcal{P}(Z,-K),\eqslantless)$, respectively.}
\end{proposition}
\begin{proof}
  Let $A,B \in \mathcal{P}(Z,K)$.
  Since $A+\hat{K}\subseteq A$, it follows $A \subseteq A^\bullet$, that is, $A^\bullet \leqslant A$.
  In particular, $(A^\bullet)^\bullet \leqslant A^\bullet$.
  Conversely, from $(A^\bullet)^\bullet+\hat{K}/2\subseteq A^\bullet$, we get $(A^\bullet)^\bullet+\hat{K}\subseteq A^\bullet+\hat{K}/2\subseteq A$, that is  $(A^\bullet)^\bullet\subseteq A^\bullet$ proving the reverse inclusion.
  Further, if $A \leqslant B$, then $B^\bullet +\hat{K}\subseteq B\subseteq A$, which implies $B^\bullet\subseteq A^\bullet$, that is $A^\bullet \leqslant B^\bullet$.
  Finally, since $(\sup_j A_j)^\bullet\subseteq A_i^\bullet$ for all $i$, one has $(\sup_i A_i)^\bullet\subseteq \sup_i A_i^\bullet$ and conversely, $A^\bullet_i +\hat{K}\subseteq A_i$ for all $i$ implies $\sup_i A^\bullet_i+\hat{K}\subseteq \sup_i A_i$, showing that $\sup_i A^\bullet_i\subseteq (\sup_i A_i)^\bullet$.
\end{proof}

According to the previous proposition, the supremum of a family of elements in ${\cal K}(Z,K)$ stays in ${\cal K}(Z,K)$.
However, the infimum $\inf_i A_i^\bullet$ is in general not in ${\cal K}(Z,K)$.
\begin{proposition}
    The space $(\mathcal{K}(Z, K),\leqslant)$ is a complete lattice for the lattice operations
    \begin{equation}\inf A_i :=\left( \bigcup A_i \right)^\bullet\quad \text{and}\quad 
        \sup A_i:=\bigcap A_i \label{eq:passage03}
    \end{equation}
    for every family $(A_i)\subseteq \mathcal{K}(Z,K)$. 
    
    Analogously, the space $(\mathcal{K}(Z,-K),\eqslantless)$ is a complete lattice for the lattice operations
    \begin{equation}
        \inf A_i :=\bigcap A_i\quad\text{and}\quad \sup A_i :=\left(\bigcup A_i\right)^\bullet\label{eq:passage04}  
    \end{equation}
    for every family $(A_i)\subseteq \mathcal{K}(Z,-K)$.
\end{proposition}
\begin{proof}
    It is enough to show that $\inf A_i$ in \eqref{eq:passage03} is a minimal element of the family $(A_i)$ for $\leqslant$ in $\mathcal{K}(Z, K)$.
    Since $\cup A_i\leqslant A_i$ for all $i$, it follows from Proposition \ref{lem:passage} that $\inf A_i=(\cup A_i)^\bullet \leqslant A_i^\bullet=A_i$ for all $i$.
    Further, for $A \in \mathcal{K}(Z, K)$ such that $\inf A_i \leqslant A\leqslant A_i$ for all $i$, it follows that $\inf A_i \leqslant A\leqslant \cup A_i$ so that $\inf A_i\leqslant A=A^\bullet \leqslant (\cup A_i)^\bullet =\inf A_i$, showing that $A=\inf A_i$.
\end{proof}

\subsection{Increasing, Convex, and Quasiconvex Functions}
Let $X:=(X,\leq)$ and $Z=(Z,\leq)$ be preordered vector spaces with $C=\set{c\in X\colon  c\geq 0}$ and $K=\set{k\in Z\colon  k\geq 0}$, respectively.
\begin{definition}
  Let $F\colon X \to \mathcal{P}(Z,K)$.
  The \emph{upper level sets}, and the \emph{lower level sets} of $F$ at level $z\in Z$ are defined as 
  \begin{equation*}
    \mathcal{L}_F(z)=\Set{x \in X\colon F(x)\leqslant z} \quad\text{ and }\quad \mathcal{U}_F(z)=\Set{x \in X\colon z \leqslant F(x)},
  \end{equation*}
  respectively.
  The \emph{epigraph} and \emph{hypograph} of $F$ are defined as
  \begin{equation*}
    \epi F=\Set{(x,z)\in X\times Z\colon F(x)\leqslant z}\quad \text{and}\quad\hypo F=\Set{(x,z)\in X \times Z\colon  z\leqslant F(x)},
  \end{equation*}
  respectively.
  We further say that 
  \begin{itemize}
    \item $F$ is \emph{increasing} if $F(x)\leqslant F(y)$ whenever $x,y \in X$ and $x\leq y$;
    \item $F$ is \emph{quasiconvex} if $ F\left(\lambda x+(1-\lambda)y\right)\leqslant \sup \set{F\left(x\right), F\left(y\right)}$, for all $x,y\in X$ and $\lambda \in (0,1)$;
    \item $F$ is \emph{quasiconcave} if $\inf\set{F(x),F(y)}\leqslant F\left(\lambda x+(1-\lambda)y\right)$	for all $x,y\in X$ and $\lambda \in (0,1)$;
    \item $F$ is \emph{convex} if $F\left( \lambda x +\left( 1-\lambda \right)y\right)\leqslant \lambda F\left( x \right)+\left( 1-\lambda \right)F\left( y \right)$	for all $x,y\in X$ and $\lambda \in (0,1)$;
    \item $F$ is \emph{concave} if $\lambda F\left( x \right)+\left( 1-\lambda \right)F\left( y \right)\leqslant F\left( \lambda x +\left( 1-\lambda \right)y\right)$	for all $x,y\in X$ and $\lambda \in (0,1)$.
  \end{itemize} 
  Finally, the function $F^{-1}\colon Z \to \mathcal{P}(X)$, $z \mapsto F^{-1}(z):=\Set{x \in X\colon z \in F(x)}$ is called the inverse of $F$.
  The same definitions also apply to $F\colon X \to \mathcal{P}(Z,-K)$ with the corresponding relation $\eqslantless$.
\end{definition}
\begin{remark}\label{rem:fdslkjhi}
  For the inverse of a function $F\colon X\to{\cal P}(Z)$ one has $(F^{-1})^{-1}=F$, since $z \in F(x)$ exactly when $x \in F^{-1}(z)$.
  If in addition $F\colon X \to \mathcal{P}(Z,K)$, then it holds $\mathcal{L}_F(z)=F^{-1}(z)$ for all $z \in Z$ and $\epi F=\Set{(x,z)\in X\times Z\colon z \in F(x)}$,  due to Remark \ref{rem:blabla01}.
  Analogously, if $F\colon X\to \mathcal{P}(Z,-K)$, then $\mathcal{U}_F(z)=F^{-1}(z)$ for all $z \in Z$, and $\hypo F=\Set{(x,z)\in X\times Z\colon z\in F(x)}$. 
\end{remark}

\begin{proposition}\label{prop:quasiconvex}
  For $F\colon X\to \mathcal{P}(Z,K)$, the following statements hold:
  \begin{enumerate}[label=(\roman*)]
    \item\label{prop:quasiconvex01} $F$ is quasiconvex if, and only if, $\mathcal{L}_F(z)=F^{-1}(z)$ is convex for all $z \in Z$.
    \item \label{prop:quasiconvex02} $F$ is convex if, and only if, $\epi F$ is convex.
    \item \label{prop:quasiconvex03} if $F$ is convex, then $F$ is quasiconvex and $F(x)$ is convex for all $x \in X$.
    \item \label{prop:quasiconvex04} $F$ is quasiconcave, if, and only if, $\mathcal{L}^c_F(z)$ is convex for all $z \in Z$.
      If $F$ takes values in $\mathcal{K}(\R,\R_+)$, then $F$ is quasiconcave if, and only if, $\mathcal{U}_{F}(z)$ 
      is convex for every $z \in Z$.
    \item \label{prop:quasiconvex05} If $F$ is concave, then $\hypo F$ is convex.
      The converse statement holds true if $F$ takes values in $\mathcal{K}(\R,\R_+)$.
  \end{enumerate}
\end{proposition}
\begin{proof}
  \begin{enumerate}[label=\textit{(\roman*)},fullwidth]
    \item Suppose that $F$ is quasiconvex.
      Given $z \in Z$, let $x^1,x^2 \in F^{-1}(z)$ for which $F(x^1) \leqslant z$ and $F(x^2)\leqslant z$ so that $\sup \set{F(x^1), F(x^2)}\leqslant z$.
      From the quasiconvexity, $F(\lambda x^1+(1-\lambda)x^2)\leqslant z$, and so $\lambda x^1+(1-\lambda)x^2 \in F^{-1}(z)$ for all $\lambda \in(0,1)$, that is $F^{-1}(z)$ is convex.

      On the other hand, let $x^1,x^2 \in X$, $\lambda \in (0,1)$ and $z \in Z$ such that $\sup\set{F(x^1), F(x^2)}\leqslant z$.
      It follows $F(x^i)\leqslant  z$, for $i=1,2$.
      From the convexity of $F^{-1}(z)$, we deduce $F(\lambda x^1+(1-\lambda)x^2)\leqslant z$ for all $z \in Z$ satisfying $\sup\set{F(x^1),F(x^2)}\leqslant z$.
      Hence, $F(\lambda x^1+(1-\lambda)x^2)\leqslant \sup\set{F(x^1), F(x^2)}$.

    \item Suppose $F$ is convex.
      For $(x^i,z^i) \in \epi F$, $i=1,2$, and $\lambda \in (0,1)$, it holds $F(x^i)\leqslant z^i$ for $i=1,2$.
      Hence, $\lambda F(x^1)+(1-\lambda)F(x^2)\leqslant \lambda z^1 +(1-\lambda)z^2$ and the convexity of $F$ implies $F(\lambda x^1+ (1-\lambda)x^2)\leqslant \lambda z^1 +(1-\lambda)z^2$, which shows that $\lambda(x^1,z^1)+(1-\lambda)(x^2,z^2)\in \epi F$.

      Conversely, fix $x^1,x^2 \in X$ and $\lambda \in (0,1)$ and notice that $(x^i,z^i)\in \epi F$, for any $F(x^i)\leqslant z^i$, $i=1,2$.
      By convexity of the epigraph holds $(\lambda x^1+(1-\lambda)x^2,\lambda z^1+(1-\lambda)z^2)\in \epi F$, which implies $F(\lambda x^1+(1-\lambda)F(x^2))\leqslant \lambda z^1+(1-\lambda)z^2$ for any $F(x^i)\leqslant z^i$.
      Thus, $F(\lambda x^1+(1-\lambda)F(x^2))\leqslant \lambda F(x^1)+(1-\lambda)F(x^2)$.

    \item Let $F$ be convex so that $\epi F$ is convex by \textit{\ref{prop:quasiconvex02}}. If $(x^1,z),(x^2,z)\in \epi F$, then convexity of $\epi F$ yields $(\lambda x^1+(1-\lambda)x^2,z)\in \epi F$.
      Hence, $F^{-1}(z)$ is convex which by means of \textit{\ref{prop:quasiconvex01}} implies that $F$ is quasiconvex.

      To show that $F(x) $ is convex, let $z^1,z^2 \in F(x)$ and $\lambda \in (0,1)$.
      Then $(x,z^1),(x,z^2) \in \epi F$, and since $\epi F$ is convex, $(x,\lambda z^1+(1-\lambda)z^2) \in \epi F$, that is $\lambda z^1 (1-\lambda)z^2 \in F(x)$.

    \item Given $z \in Z$, let $x^1,x^2 \in \mathcal{L}^c_{F}(z)$ and $\lambda \in (0,1)$.
      It follows that $z \not \in F(x^1)$ and $z \not \in F(x^2)$ and therefore $z \not \in \inf\set{F(x^1),F(x^2)}$.
      Since $F$ is quasiconcave, it follows that $z \not \in F(\lambda x^1+(1-\lambda)x^2)$ and therefore $\lambda x^{1} +(1-\lambda)x^2 \in \mathcal{L}^c_{F}(z)$.
      Conversely, for every $z \not \in \inf\set{F(x^1),F(x^2)}$, it follows that $z \not \in F(x^1)$ and $z \not \in F(x^2)$.
      Hence $x^1,x^2 \in \mathcal{L}^c_F(z)$.
      By convexity of $\mathcal{L}^c_F(z)$, it follows that $\lambda x^1+(1-\lambda)x^2 \in \mathcal{L}^c_F(z)$ and therefore $z \not \in F(\lambda x^1+(1-\lambda)x^2)$.
      Thus $[F(\lambda x^1+(1-\lambda)x^2)]^c\subseteq [F(x^1)]^c\cap [F(x^2)]^c=[\inf\set{F(x^1),F(x^2)}]^c$ which implies $\min\set{F(x^1),F(x^2)}\leqslant F(\lambda x^1+(1-\lambda)x^2)$.

      Suppose now that $F$ takes values in $\mathcal{K}(\R,\R_+)$, that is, $F(x)=[\inf F(x),+\infty[$.
      The lower and upper level sets are given by ${\cal L}_F(z)=\set{x\in X\colon \inf F(x)\le z}$ and ${\cal U}_F(z)=\set{x\in X\colon  z< \inf F(x)}$. Thus, $X={\cal L}_F(z)+{\cal U}_F(z)$ , that is ${\cal U}_F(z)={\cal L}_F^c(z)$ for all $z\in Z$.
    
    \item Suppose that $F$ is concave. Given $(x^1,z^1),(x^2,z^2)\in \hypo F$ and $\lambda \in (0,1)$.
      It follows that $z^1 \leqslant F(x^1)$ and $z^2 \leqslant F(x_2)$ and therefore $\lambda z^1+(1-\lambda)z^2 \leqslant \lambda F(x^1)+(1-\lambda)F(x^2)$.
      It follows that $\lambda z^1 +(1-\lambda) z^2\leqslant F(\lambda x^1+(1-\lambda)x^2 )$ by concavity of $F$, so that $\lambda (x^1,z^1)+(1-\lambda)(x^2,z^2) \in \hypo F$.

      Suppose now that $F$ maps into $\mathcal{K}(\R,\R_+)$.
      Let $z \leqslant F(x^1)$ and $z \leqslant F(x^2)$ so that $z \leqslant \lambda F(x^1)+(1-\lambda)F(x^2)$.
      Since $\hypo F$ is convex, it follows that $z \leqslant F(\lambda x^1+(1-\lambda)x^2)$.
      The same argumentation as previously shows
      \begin{equation*}
        \inf \set{\lambda F(x^1)+(1-\lambda)F(x^2)}=\lambda \inf F(x^1)+(1-\lambda)\inf F(x^2)\leq \inf F(\lambda x^1 +(1-\lambda)x^2)
      \end{equation*}
      and therefore $\lambda F(x^1)+(1-\lambda)F(x^2)\leqslant F(\lambda x^1+(1-\lambda)x^2)$.
  \end{enumerate}
\end{proof}
\begin{proposition}\label{prop:inverse}
  Let $F\colon X\to \mathcal{P}(Z,K)$ be an increasing function. Then the inverse $F^{-1}$ maps from $Z$ to $\mathcal{P}(X,-C)$ and is increasing.
  Furthermore, $F$ is convex if, and only if, $F^{-1}$ is concave and if, and only if, $\epi F$ and $\hypo F^{-1}$ are convex.
\end{proposition}
\begin{proof}
  For  $x \in F^{-1}(z)$ and $c\in C$ it holds $F(x-c)\leqslant F(x) \leqslant z$ and therefore $F^{-1}(z)-C\subseteq F^{-1}(z)$, that is, $F^{-1}(z)\in {\cal P}(X,-C)$.
  Fix now $z_1\le z_2$ so that $z_1\in F(x)$ implies $z_2\in F(x)$.
  Hence
  \begin{equation*}
    F^{-1}(z_1)=\set{x\in X\colon  z_1\in F(x)}\subseteq \set{x\in X\colon  z_2\in F(x)}=F^{-1}(z_2),
  \end{equation*}
  which shows that $F^{-1}(z_1)\eqslantless F^{-1}(z_2)$ showing that $F^{-1}$ is increasing.

  Suppose that $F$ is convex.
  For $(x^i,z^i) \in \epi F$, $i=1,2$, and $\lambda \in (0,1)$, it holds $F(x^i)\leqslant z^i$ for $i=1,2$.
  Hence, $\lambda F(x^1)+(1-\lambda)F(x^2)\leqslant \lambda z^1 +(1-\lambda)z^2$ and the convexity of $F$ implies $F(\lambda x^1+ (1-\lambda)x^2)\leqslant \lambda z^1 +(1-\lambda)z^2$, which shows that $\lambda(x^1,z^1)+(1-\lambda)(x^2,z^2)\in \gr F$.
  Since $\hypo F^{-1}$ is up to flipping the coordinates equal to $\epi F$, we deduce that $\hypo F^{-1}$ is also convex.

  Conversely, fix $x^1,x^2 \in X$ and $\lambda \in (0,1)$ and notice that $(x^i,z^i)\in \epi F$, for any $F(x^i)\leqslant z^i$, $i=1,2$.
  By convexity of the epigraph holds $(\lambda x^1+(1-\lambda)x^2,\lambda z^1+(1-\lambda)z^2)\in \gr F$, which implies $F(\lambda x^1+(1-\lambda)F(x^2))\leqslant \lambda z^1+(1-\lambda)z^2$ for any $F(x^i)\leqslant z^i$.
  Thus, $F(\lambda x^1+(1-\lambda)F(x^2))\leqslant \lambda F(x^1)+(1-\lambda)F(x^2)$.

  Let us show that if $F$ is convex, then $F^{-1}$ is concave.
  For $\lambda \in (0,1)$ and $z^1,z^2 \in Z$, let $x^i \in F^{-1}(z^i)$ for $i=1,2$ so that $x=\lambda x^1+(1-\lambda)x^2 \in \lambda F^{-1}(z^1)+(1-\lambda)F^{-1}(z^2)$.
  By means of Proposition \ref{prop:inverse}, it follows that $z^i \in F(x^i)$ for $i=1,2$ and therefore $\lambda z^1+(1-\lambda)z^2 \in \lambda F(x^1)+(1-\lambda) F(x^2)$.
  However, $F$ being convex, it follows that $F(\lambda x^1+(1-\lambda)x^2)\leqslant \lambda z^1+(1-\lambda)z^2$, which by means of Proposition \ref{prop:inverse} yields $\lambda x^1+(1-\lambda)x^2 \in F^{-1}(\lambda z^1+(1-\lambda)z^2 )$.
  Hence, $\lambda F^{-1}(z^1)+(1-\lambda)F^{-1}(z^2)\eqslantless F^{-1}(\lambda z^1+(1-\lambda)z^2)$.
  The converse statement is analogous.
\end{proof}

\begin{remark}\label{rem:quasiconcave}
  Proposition \ref{prop:inverse} also holds when $K$ and $-C$ are replaced by $-K$ and $C$, and the respective orders
  $\leqslant$ by $\eqslantless$ and $\eqslantless$ by $\leqslant$, respectively.

  For $F\colon X \to \mathcal{P}(Z,-K)$, the statements of Proposition \ref{prop:quasiconvex} modify as follows
  \begin{enumerate}[label=(\roman*)]
    \item\label{prop:quasiconcave01} $F$ is quasiconcave if, and only if, $\mathcal{U}_{F}(z)=F^{-1}(z)$ is convex for all $z \in Z$;
    \item \label{prop:quasiconcave02} $F$ is concave if, and only if, $F^{-1}$ is convex and if, and only if, $\epi F^{-1}$ and $\hypo F$ are convex;
    \item \label{prop:quasiconcave03} if $F$ is concave, then $F$ is quasiconcave and $F(x)$ is convex for any $x \in X$;
    \item \label{prop:quasiconcave04} $F$ is quasiconvex, if, and only if, $\mathcal{U}^c_{F}(z)$ is convex for all $z \in Z$.
      If $F$ takes values in $\mathcal{K}(\R,-\R_+)$, then $F$ is quasiconcave if, and only if, $\mathcal{L}_{F}(z)$ is convex for every $z \in Z$;
    \item \label{prop:quasiconcave05} if $F$ is convex, then $\epi F$ is convex.
      The reciprocal is true if $F$ takes values in $\mathcal{K}(\R,-\R_+)$.

  \end{enumerate}
  In particular, if $F\colon X \to \mathcal{P}(Z,K)$ is increasing and convex, then $F^{-1}\colon Z\to \mathcal{P}(X,-C)$ is increasing and concave.
\end{remark}

\section{Duality of Increasing Functions}\label{sec:02}
Throughout this section, let $X=(X,\leq)$ and $Z=(Z,\leq)$ be preordered vector spaces with $C=\set{c \in X\colon  c\geq 0}$ and $K=\set{k \in Z\colon  k\geq 0}$, respectively.
Furthermore, we assume that $\hat{C}=\set{c \in C\colon  c>0}$ as well as $\hat{K}=\set{k\in K\colon  k>0}$ are non-empty.
Given an increasing function $F\colon X\to \mathcal{K}(Z,K)$, we define its \emph{left-} and \emph{right-continuous version}\footnote{Since $\hat{C}\neq \emptyset$, it follows that $\set{y \in X\colon y<x}$ and $\set{y\in X\colon x< y}$ are non-empty for every $x \in X$.} $F^-\colon  X\to \mathcal{K}(Z,K)$ and $F^+\colon X\to \mathcal{K}(Z,K)$ by
\begin{equation*}
  F^-(x) :=\sup_{y<x} F(y)\qquad \text{and}\qquad F^+(x): =\inf_{x<y} F(y).
\end{equation*}
We say that an increasing function $F$ is \emph{left-continuous} or \emph{right-continuous} if $F=F^-$ or $F=F^+$ respectively.
Since $F$ is increasing it follows immediately that $F^-\leqslant F\leqslant F^+$.
\begin{lemma}\label{lem:LFRF}
  Let $F\colon X \to \mathcal{K}(K,Z)$ be an increasing function.\footnote{The statement of the Lemma also holds for increasing functions $F\colon Z\to \mathcal{K}(X,-\hat{C})$ replacing $\leqslant$ by $\eqslantless$.}
  Then
  \begin{equation}
    F^+(x)\leqslant F^-(y),\quad \text{ for all }x<y,
    \label{eq:LFRF01}
  \end{equation}
  and
  \begin{equation}
    F^-=(F^-)^-=(F^+)^-,\quad\text{and}\quad F^+=(F^+)^+=(F^-)^+.
    \label{eq:LFRF02}
  \end{equation}
\end{lemma}
\begin{proof}
    As for \eqref{eq:LFRF01}, it follows directly from the definition and $\hat{K}\neq \emptyset$ since $F^+(x)\leqslant F(z)\leqslant F^-(y)$ for every $z \in \set{\tilde{z}\colon  x<\tilde{z}<y}$.
    As for \eqref{eq:LFRF02}, on the one-hand, $F^-\leqslant F\leqslant F^+$ yields $(F^-)^-\leqslant F^-\leqslant (F^+)^-$.
    On the other-hand, \eqref{eq:LFRF01} yields $(F^+)^-(x)=\sup_{\tilde{x}<x} F^+(\tilde{x})\leqslant \sup_{\tilde{x}<x}F^-(\tilde{x})=(F^-)^-(x)$ so that $(F^+)^-=(F^-)^-$ which shows that $(F^-)^-= F^- = (F^+)^-$.
    Relation $F^+=(F^+)^+=(F^-)^+$ follows along the same argumentation.
\end{proof}
Let now $F\colon X \to \mathcal{K}(Z,K)$ and $G\colon Z \to \mathcal{K}(X,-C)$ be increasing functions.
By Proposition \ref{prop:inverse}, $F^{-1}\colon Z \to \mathcal{P}(X,-C)$ and $G^{-1}\colon X\to \mathcal{P}(Z,K)$ are both increasing.
The main result of this section is a complete duality between increasing functions.
\begin{theorem}\label{thm:onetoone}
  The inverse $F^{-1}$ of any increasing left-continuous function $F\colon X \to \mathcal{K}(Z,K)$ is an increasing right-continuous function from $Z$ to $\mathcal{K}(X,-C)$.
  Conversely, the inverse $G^{-1}$ of any increasing right-continuous function $G\colon Z \to \mathcal{K}(X,-C)$ is an increasing left-continuous function from $X$ to $\mathcal{K}(Z,K)$.

  Further, for any increasing left-continuous function $F\colon X \to \mathcal{K}(Z,K)$, and increasing right-continuous function $G\colon  Z \to \mathcal{K}(Z,-C)$, it holds
  \begin{equation}\label{eq:eqonetoone}
    F=(F^{-1})^{-1}\quad \text{and}\quad G=(G^{-1})^{-1}.
  \end{equation}
  If $G=F^{-1}$ we have\footnote{Recall that with our conventions of Remark \ref{rem:blabla01}, Relation \eqref{eq:C2} corresponds to 
  $F^+(x)\leqslant \tilde{z}$ for some $\tilde{z}<z$ if, and only if, $\tilde{x}\eqslantless G^-(z)$ for some $\tilde{x}>x$.  
  }
  \begin{align}
    &&F(x)&\leqslant z &&\text{if, and only if,}&  x&\eqslantless G(z)&&\label{eq:C1}\\
    &&F^+(x) &< z &&\text{if, and only if,}&  x &< G^-(z)&&\label{eq:C2}
  \end{align}
\end{theorem}

\begin{proof}
  Let $F\colon X \to \mathcal{K}(Z,K)$ be an increasing left-continuous function.
  By Proposition \ref{prop:inverse}, $F^{-1}$ is increasing and maps into $\mathcal{P}(X,-C)$.
  Since $F$ is left-continuous, it further holds
  \begin{multline*}
    \left[F^{-1}(z)\right]^\bullet=\Set{x \in X\colon  x-c \in F^{-1}(z)\text{ for all }c>0}=\Set{x \in X\colon  z \in F(x-c)\text{ for all }c>0}\\
    =\Set{x \in X\colon  z \in \sup_{\tilde{x}<x}F(\tilde{x})=F^-(x)=F(x)}=\Set{x \in X\colon  z \in F(x)}=F^{-1}(z).
  \end{multline*}
  Hence $F^{-1}$ is an increasing function from $Z$ to $\mathcal{K}(X,-C)$.
  As for the right-continuity, since $F$ maps to $\mathcal{K}(Z,K)$, it follows that $F(x)=[F(x)]^\bullet$, so that
  \begin{multline*}
    \left[F^{-1}\right]^+(z)=\inf_{z<\tilde{z}}F^{-1}(\tilde{z})=\Set{x \in X\colon  x \in F^{-1}(z+k)\text{ for all }k>0}\\
    =\Set{x \in X\colon  z+k \in F(x)\text{ for all }k>0}=\Set{x \in X\colon  z \in [F(x)]^\bullet=F(x)}=F^{-1}(z),
  \end{multline*}
  showing that $F^{-1}$ is right-continuous.
  As for \eqref{eq:eqonetoone}, it follows from $z \in F(x)$ if, and only if, $x \in F^{-1}(z)$ as previously mentioned in Remark \ref{rem:fdslkjhi}.
  Relation \eqref{eq:C1} follows from $F(x)\leqslant z$ if, and only if, $z \in F(x)$ if, and only if, $x \in F^{-1}(z)=G(z)$ if, and only if, $x \eqslantless G(z)$.
  Finally, Relation \eqref{eq:C2} follows from \eqref{eq:C1} and Lemma \ref{lem:LFRF}.
\end{proof}

\begin{lemma}\label{lem:minmax}
  Let $F_1,F_2\colon  X \to \mathcal{K}(Z,K)$ be increasing left-continuous functions and $G_1,G_2\colon Z \to \mathcal{K}(X,-C)$ be increasing increasing right-continuous functions, such that $F_i^{-1}=G_i$ for $i=1,2$.
  Then
  \begin{equation}
    F_1\leqslant F_2\quad \text{if, and only if,} \quad G_2 \eqslantless G_1.
    \label{eq:minimality}
  \end{equation}
\end{lemma}
\begin{proof}
  Suppose that $F_1\leqslant F_2$,
  \begin{multline*}
    G_2(z)=F_2^{-1}(z)=\Set{x \in X\colon z \in F_{2}(x)}=\Set{x \in X\colon F_{2}(x)\leqslant z}\\
    \eqslantless\Set{x \in X\colon F_{1}(x)\leqslant z}=\Set{x \in X\colon  z \in F_1(x)}=F_1^{-1}(z)=G_1(z)	
  \end{multline*}
  and therefore $G_{2}\eqslantless G_1$.
  The converse implication follows along the same argumentation.
\end{proof}

\section{Topological Duality of Quasiconvex-Functions}\label{sec:03}

In this section we assume that $X$ is a locally convex topological vector space with dual $X^\ast$.
The dual pairing between $x\in X$ and $x^\ast \in X^\ast$ is denoted by $\langle x^\ast,x\rangle$.
Further, $Z$ is a preordered vector space with convex ordering cone $K=\set{k \in Z\colon  k\geq 0}$.
A function $F\colon X\to\mathcal{K}(Z,K)$ is called \emph{lower level-closed} if $\mathcal{L}_{F}(z)=\set{x \in X\colon  F(x)\leqslant z}$ is closed for all $z \in Z$.\footnote{Analogously, $G\colon X \to \mathcal{K}(Z,-K)$ is called \emph{lower level-closed} if $\mathcal{L}_{G}(z)=\set{x \in X\colon  G(x)\geqslant z}$ is closed for all $z \in Z$.}

The existence of a dual representation for lower level-closed quasiconvex functions is rather straightforward and does not need any assumption on the preordered vector space $Z$.
\begin{proposition}\label{prop:existence}
    Let $F\colon X \to \mathcal{P}(Z,K)$ be a lower level-closed quasiconvex function.
    Then
    \begin{equation}\label{eq:sdlkfji}
        F(x)=\sup_{x^\ast \in X^\ast} R(x^\ast, \langle x^\ast,x\rangle), \quad x \in X
    \end{equation}
    where $R\colon X^\ast \times \mathbb{R}\to \mathcal{P}(Z,K)$ is given by
    \begin{equation*}
        R(x^\ast, s)=\Set{z \in Z\colon  s\leq \sup_{x \in F^{-1}(z)} \langle x^\ast, x\rangle }.
    \end{equation*}
\end{proposition}
\begin{proof}
    Since $F^{-1}(z)$ is closed and convex, the Hahn-Banach separation theorem yields
    \begin{multline*}
        F(x)=\left(F^{-1}\right)^{-1}(x)=\Set{z \in Z\colon  x \in F^{-1}(z) }\\
        =\Set{z \in Z\colon  \langle x^\ast,x\rangle\leq \sup_{x \in F^{-1}(z)} \langle x^\ast, x\rangle,\text{ for all }x^\ast \in X^\ast }\\
        =\sup_{x^\ast \in X^\ast}\Set{z \in Z\colon  \langle x^\ast,x\rangle\leq \sup_{x \in F^{-1}(z)} \langle x^\ast, x\rangle }=\sup_{x^\ast \in X^\ast} R(x^\ast, \langle x^\ast,x\rangle).
    \end{multline*}
    Finally, let $z \in R(x^\ast,s)$ and $k \geq 0$.
    By Proposition \ref{prop:inverse} with $C=\set{0}$, it follows that $F^{-1}(z) \eqslantless F^{-1}(z+k)$.
    Hence, $s\leq \sup_{x \in F^{-1}(z)}\langle x^\ast, x\rangle\leq \sup_{x \in F^{-1}(z+k)}\langle x^\ast, x\rangle$ showing that $R(x^\ast, s)\in \mathcal{P}(Z,K)$.
\end{proof}

The goal is to characterize uniquely the dual function $R$ in \eqref{eq:sdlkfji}.
It turns out that such a unique characterization can be achieved for lower level-closed and quasiconvex functions $F\colon X \to \mathcal{K}(Z,K)$.
So, from now on, we assume that $\hat{K}=\set{k \in Z\colon k>0}\neq \emptyset$.
A \emph{risk function}\footnote{See \citet{drapeau01} for a justification of this denomination.} is a function $R\colon  X^\ast\times \mathbb{R}\to \mathcal{K}(Z, K)$ such that
\begin{enumerate}[label=(\roman*)]
    \item \label{cond:max01} $s\mapsto R(x^\ast,s)$ is increasing and left-continuous for every $x^\ast \in X^\ast$.
\end{enumerate}
We denote by $\mathcal{R}$ the set of all risk functions.
Given a risk function $R\in \mathcal{R}$, the function defined as
\begin{equation*}
    F(x)\colon =\sup_{x^\ast \in X^\ast} R(x^\ast, \langle x^\ast,x\rangle), \quad x \in X
\end{equation*}
is a lower level-closed quasiconvex function from $X$ to $\mathcal{K}(Z,K)$.
If $Z$ is a properly preordered vector space, we call a risk function $R \in \mathcal{R}$ \emph{maximal} if 
\begin{enumerate}[label=(\roman*), resume]
    \item \label{cond:max05} $R$ is jointly quasiconcave and such that
        \begin{equation}
            R(\lambda x^\ast,s)=R(x^\ast,s/\lambda),
            \label{eq:poshomo}
        \end{equation}
        for all $x^\ast \in X^\ast$, $s\in \R$ and $\lambda>0$;
    \item \label{cond:max02} the set $\Set{x^\ast \in X^\ast\colon  R^+(x^\ast, s) <z}$ is open for every $s \in \R$ and $z \in Z$;

    \item \label{cond:max03} $\cup_{s \in \R}R(x^\ast,s)=\cup_{s \in \R}R(y^\ast,s)$ for all $x^\ast,y^\ast \in X^\ast$.
\end{enumerate}
The set of maximal risk functions is denoted by $\mathcal{R}^{\max}$.

\begin{remark}
    In case where $\leqslant$ is a total preorder, for instance $\mathcal{K}(Z,K)=\mathcal{K}(\mathbb{R},\mathbb{R}_+)$, then the set
    \begin{equation*}
        \Set{x^\ast \in X^\ast\colon  R^+(x^\ast, s) < z}^c=\Set{x^\ast \in X^\ast\colon  z\leqslant R^+(x^\ast, s) }
    \end{equation*}
    is closed.
    Thus Condition \ref{cond:max02} states that $R^+(\cdot,s)$ is  upper level-closed in accordance to the scalar characterization in \citep{drapeau01}.
\end{remark}
If $Z$ is not properly preordered, we replace Condition \ref{cond:max05} by
\begin{enumerate}[label=(\roman*${}^\prime$)]
        \setcounter{enumi}{1}
    \item \label{cond:max05bis} $R$ fulfills \eqref{eq:poshomo} and is such that 
        \begin{equation}
            \bigcap_{\tilde{z}<z}\mathcal{L}^c_{ R^+}(\tilde{z})=\bigcap_{\tilde{z}<z}\Set{(x^\ast,s) \in X^\ast\times \R\colon  \tilde{z}\not \in R^+(x^\ast,s)},
        \end{equation}
        is convex for all $z\in Z$;
\end{enumerate}
and still keep the notation $\mathcal{R}^{\max}$ as well as the denomination maximal risk function.
Our main theorem reads as follows.
\begin{theorem}\label{thm:robrep}
    Let $(Z,\leq)$ be a preordered vector space such that $\hat{K}=\set{k\in Z\colon  k>0}\neq \emptyset$.
    Any lower level-closed and quasiconvex function $F\colon X \to \mathcal{K}(Z,K)$ admits the dual representation
    \begin{equation}
        F\left(x\right)= \sup_{x^\ast \in  X^\ast } R\left(x^\ast,\langle x^\ast,x\rangle\right),
        \label{thm:main1}
    \end{equation}
    for a unique $R \in \mathcal{R}^{\max}$.\footnote{From our convention, if $Z$ is additionally properly preordered, Condition \ref{cond:max05} is taken into account instead of \ref{cond:max05bis}}

    Furthermore, if \eqref{thm:main1} holds for another risk function $\tilde{R} \in \mathcal{R}$, then $\tilde{R}\leqslant R$.
\end{theorem}
\begin{remark}
    The second assertion in Theorem \ref{thm:robrep} justifies the term ``maximal'' risk function: It gives an alternative criterion of uniqueness in terms of maximality within the class of risk functions.
\end{remark}
\begin{remark}
    In the scalar case, \citep[Proposition 6, p.40]{drapeau01} states a one-to-one relation between further properties of $F$---such as translation invariance, positive homogeneity, scaling invariance, etc.---and the respective ones of the maximal risk function $R$.
    Analogue results can be obtained in the set-valued case adapting the proof argumentation in \citep{drapeau01} to the set-valued techniques of the present paper.
    However, an introduction and discussion of these properties in connection with applications in Economics and Finance are best discussed in a separate paper.
\end{remark}
Before addressing the proof of Theorem \ref{thm:robrep}, we introduce minimal penalty functions, which are dual to maximal risk functions, and link them to the support functions of $F^{-1}(z)$.
%
%

\subsection{Minimal Penalty and Maximal Risk Functions}
If $Z$ is properly preordered, a \emph{minimal penalty function} is a function $\alpha\colon  X^\ast \times Z\to \mathcal{K}(\mathbb{R},-\mathbb{R}_{+})$ such that 
\begin{enumerate}[label=(\alph*)]
    \item \label{cond:min01} $z\mapsto \alpha(x^\ast,z)$ is increasing and right-continuous for every $x^\ast \in X^\ast$;
    \item \label{cond:min05} $\alpha(\lambda x^\ast,z)=\lambda \alpha(x^\ast,z)$ and $x^\ast \mapsto \alpha(x^\ast,z)$ is convex for every $z\in Z$ and $\lambda >0$;
    \item \label{cond:min02} $x^\ast \mapsto \alpha^-(x^\ast,z)$ is lower level-closed, for every $z\in Z$;
    \item \label{cond:min03} $\alpha^-\left( x^\ast,z \right)=\emptyset$ for some $x^\ast \in X^\ast$ and $z \in Z$, then $\alpha^-\left( \cdot,z \right)\equiv\emptyset$ on $ X^\ast$.
\end{enumerate}
The set of minimal penalty functions is denoted by $\mathcal{P}^{\min}$.
If $Z$ is not properly preordered, we replace Condition \ref{cond:min05} by
\begin{enumerate}[label=(\alph*${}^\prime$)]
        \setcounter{enumi}{1}
    \item\label{cond:min05bis} $\alpha(\lambda x^\ast,z)=\lambda \alpha(x^\ast,z)$ and $x^\ast \mapsto \alpha^-(x^\ast,z)$ is convex for every $z\in Z$ and $\lambda >0$;
\end{enumerate}
and here also keep the notation $\mathcal{P}^{\min}$ as well as the denomination minimal penalty function.

\begin{proposition}\label{prop:onetooneRalpha}
    The inverse in the second argument states a one-to-one relation between $\mathcal{R}^{\max}$ and $\mathcal{P}^{\min}$.
    In other words, 
    \begin{align*}
        (x^\ast, z) \mapsto [\alpha(x^\ast,\cdot)]^{-1}(z) &\in \mathcal{R}^{\max}\quad \text{for every }\alpha \in \mathcal{P}^{\min};\\
        (x^\ast, s )\mapsto [R(x^\ast,\cdot)]^{-1}(s) &\in \mathcal{P}^{\min}\quad \text{for every }R \in \mathcal{R}^{\max}.
    \end{align*}
\end{proposition}
\begin{proof}
    Let $\alpha \in \mathcal{P}^{\min}$ and $R \in \mathcal{R}^{\max}$.
    To simplify notations, we denote by $\alpha^{-1}$ and $R^{-1}$ the inverse in the second argument of $\alpha$ and $R$ respectively.
    From Theorem \ref{thm:onetoone} it follows that $\alpha$ fulfills \ref{cond:min01} if, and only if, $\alpha^{-1}$ fulfills \ref{cond:max01}.
    And therefore $R$ fulfills \ref{cond:max01} if, and only if, $R^{-1}$ fulfills \ref{cond:min01}.
    Hence, we just have to show that $\alpha \in \mathcal{P}^{\min}$ if, and only if, $R=\alpha^{-1} \in \mathcal{R}^{\max}$.

    \begin{enumerate}[label=\textit{Step \arabic*:},fullwidth]

        \item Equivalence between \ref{cond:min05} and \ref{cond:max05} or \ref{cond:min05bis} and \ref{cond:max05bis}.
            First, $\alpha(\lambda x^\ast,z)=\lambda \alpha(x^\ast,z)$ for every $z$ is equivalent to
            \begin{equation*}
                s \eqslantless \alpha(\lambda x^\ast, z)\quad \text{if, and only if,}\quad s/\lambda \eqslantless \alpha(x^\ast,z)
            \end{equation*}
            for every $s\in \mathbb{R}$ and $z\in Z$.
            By means of \eqref{eq:C1}, this is equivalent to 
            \begin{equation*}
                R(\lambda x^\ast,s)\leqslant z\quad \text{if, and only if,}\quad R(x^\ast,s/\lambda)\leqslant z
            \end{equation*}
            for every $s\in \mathbb{R}$ and $z \in Z$, that is, $R(\lambda x^\ast,s)=R(x^\ast, s/\lambda)$.

            Further, in case of \ref{cond:min05} and \ref{cond:max05}, by means of Proposition \ref{prop:quasiconvex}, $\alpha(\cdot,z)$ is convex if, and only if, $\epi \alpha(\cdot,z)=\set{(x^\ast,s)\in X^\ast\times \mathbb{R}\colon  \alpha(x^\ast, z)\eqslantless s}$ is convex for every $z$.
            From the order totality of $\mathcal{K}(\mathbb{R},-\mathbb{R}_+)$, inspection shows that this holds if, and only if,
            \begin{equation*}
               \Set{(x^\ast,s)\in X^\ast\times \R\colon  \max \alpha(x^\ast,z) <s }= \Set{(x^\ast,s)\in X^\ast\times \R\colon  s \not \in \alpha(x^\ast,z) }
            \end{equation*}
            is convex for every $z$.
            Relation \ref{eq:C1}, yields
            \begin{equation*}
                \Set{(x^\ast,s)\in X^\ast\times \R\colon s \not \in \alpha(x^\ast,z)}=\Set{(x^\ast,s)\in X^\ast\times \R\colon  z \not \in R(x^\ast,s)}=\mathcal{L}^c_{R}(z)
            \end{equation*}
            which shows that $\alpha(\cdot,z)$ is convex for every $z$ if, and only if, $R$ is jointly quasiconcave.

            In case of \ref{cond:min05bis} and \ref{cond:max05bis}, using Relation \eqref{eq:C2}, it follows that
            \begin{align*}
                \epi \alpha^-(\cdot,z)&=\Set{(x^\ast,s) \in X^\ast\times \R\colon  \alpha^-(x^\ast,z)\eqslantless s}\\
                &=\Set{(x^\ast,s) \in X^\ast\times \R\colon  s< \alpha^-(x^\ast,z)}^c\\
                &=\Set{(x^\ast,s) \in X^\ast\times \R\colon R^+(x^\ast,s)<z}^c=\bigcap_{\tilde{z}<z}\mathcal{L}^c_{R^+}(\tilde{z}).
            \end{align*}
            Hence, $\alpha$ fulfills Condition \ref{cond:min05bis} if, and only if, $R$ fulfills Condition \ref{cond:max05bis}.

        \item Equivalence between \ref{cond:min02} and \ref{cond:max02}.
            The function $\alpha^-(\cdot,z)$ is lower level-closed for every $z\in Z$ if and only $\mathcal{L}_{\alpha^-(\cdot,z)}^c(s)$ is open for every $s \in \R$ and $z \in Z$.
            However, using Relation \eqref{eq:C2}, and the fact that $\eqslantless$ is total on $\mathcal{K}(\mathbb{R},-\mathbb{R}_+)$, it follows that
            \begin{align*}
                \mathcal{L}^c_{\alpha^-(\cdot,z)}(s)&=\Set{x^\ast \in X^\ast\colon  \alpha^-(x^\ast,z)\eqslantless s}^c\\
                &=\Set{x^\ast \in X^\ast\colon  s< \alpha^-(x^\ast,z)}\\
                &=\Set{x^\ast \in X^\ast\colon R^+(x^\ast,s) <z}.
            \end{align*}
            This shows that $\alpha^-(\cdot,z)$ is lower level-closed if, and only if, Condition \ref{cond:max02} holds.

        \item Equivalence between \ref{cond:min03} and \ref{cond:max03}.
            Let $A(x^\ast):=\Set{z \in Z\colon \alpha(x^\ast,z)=\emptyset}$.
            Condition \ref{cond:min03} is equivalent to $A(x^\ast)=A(\tilde{x}^\ast)$ for all $x^\ast,\tilde{x}^\ast \in X^\ast$, and hence also to $A^c(x^\ast)=A^c(\tilde{x}^\ast)$ for every $x^\ast,\tilde{x}^\ast \in X^\ast$.
            However, using Relation \eqref{eq:C1}, we obtain
            \begin{align*}
                A^c(x^\ast)&=\Set{z \in Z\colon \alpha(x^\ast,z)\neq \emptyset}=\Set{z \in Z\colon s\eqslantless \alpha(x^\ast,z)\text{ for some }s \in \R}\\
                &=\bigcup_{s \in \R}\Set{z \in Z\colon R(x^\ast,s)\leqslant z}=\bigcup_{s \in \R}R(x^\ast,s),
            \end{align*}
            which shows that \ref{cond:min03} is equivalent to \ref{cond:max03}.
    \end{enumerate}
\end{proof}

\subsection{Minimal Penalty and Support Functions}
We denote by $\sigma_{ A }\left( x^\ast \right):=\sup_{x\in A} \langle x^\ast,x\rangle$ for $x^\ast \in  X^\ast$ the support function of $A\subseteq X$ with $\sigma_\emptyset(x^\ast)=-\infty$ for every $x^\ast \in X^\ast$.
\begin{proposition}\label{prop:uniquealpha}
    Let $A\subseteq X$ be a closed convex set.
    Then, there exists a unique function $\alpha_A\colon X^\ast \to \mathcal{K}(\R,-\R_+)$, satisfying
    \begin{enumerate}
        \item \label{cond:pen03} $\alpha_A$ is convex, positively homogeneous and lower level-closed;
        \item \label{cond:pen04} if $\alpha_A\left( x^\ast \right)=\emptyset$ for some $x^\ast \in X^\ast$, then $\alpha_A\equiv\emptyset$ on $ X^\ast$;
        \item \label{cond:pen05} for all $x \in X $ holds
            \begin{equation}
                x\in A  \quad \text{if, and only if,}\quad  \langle x^\ast, x\rangle \eqslantless \alpha_A\left( x^\ast  \right)\quad\text{for all }x^\ast\in X^\ast.
                \label{cond:pen01} 
            \end{equation}
    \end{enumerate}
    This function is given by
    \begin{equation}
        \alpha_{A}(x^\ast):=\sigma_{A}(x^\ast)-\R_+,\quad x^\ast \in X^\ast.
        \label{eq:minpenfunction}
    \end{equation}
    Moreover, if $\alpha\colon X^\ast \to \mathcal{K}(\R,-\R_+)$ satisfies \eqref{cond:pen01}, then $\alpha_{A}\eqslantless \alpha$.
\end{proposition}
\begin{proof}
    Based on a classical separation argument, \citep[Lemma C.3]{drapeau01} states that $\sigma_A$ is the unique function from $X^\ast$ to $[-\infty,\infty]$ which is lower level-closed, positively homogeneous, convex, such that $\sigma_A(x^\ast)=-\infty$ for some $x^\ast \in X^\ast$ if, and only if, $\sigma_A\equiv -\infty$, that is $A=\emptyset$, and such that
    \begin{equation}\label{eq:bloblo}
        x\in A  \quad \text{if, and only if,}\quad  \langle x^\ast, x\rangle \leq \sigma_A\left( x^\ast  \right)\quad\text{for all }x^\ast\in X^\ast.
    \end{equation}
    Further, if $\sigma\colon X^\ast \to [-\infty,\infty]$ satisfies \eqref{eq:bloblo} then $\sigma_A \leq \sigma$.
    We are then left to show that $\alpha_A:=\sigma_A -\mathbb{R}_+$ satisfies Conditions \ref{cond:pen03}, \ref{cond:pen04} and \ref{cond:pen05}.
    The fact that $\alpha_A$ is positively homogeneous is immediate.
    By means of Proposition \ref{prop:quasiconvex}, the convexity follows from 
    \begin{equation*}
        \epi \alpha_A=\Set{(x^\ast,s)\in X^\ast\times \mathbb{R}\colon  \alpha_A(x^\ast)\eqslantless s}=\Set{(x^\ast,s)\in X^\ast\times \mathbb{R}\colon  \sigma_A(x^\ast)\leq s}
    \end{equation*}
    being convex.
    The lower semi-continuity follows from
    \begin{equation*}
        \mathcal{L}_{\alpha_A}(s)=\Set{x^\ast \in X^\ast\colon  \alpha_A(x^\ast)\eqslantless s}=\Set{x^\ast \in X^\ast\colon  \sigma(x^\ast)\leq s}
    \end{equation*}
    being closed for every $s \in \mathbb{R}$.
    Conditions \ref{cond:pen04} and \ref{cond:pen05} are immediate.
\end{proof}
\begin{proposition}\label{prop:minpen}
    Let $F\colon X \to \mathcal{K}(Z,K)$ be a lower level-closed quasiconvex function.
    Then, $(x^\ast ,z)\mapsto \alpha_{F^{-1}(z)}(x^\ast)$ is a function from $X^\ast \times Z$ to $\mathcal{K}(\mathbb{R},-\mathbb{R}_+)$ which is increasing in the second argument $z$.
    Furthermore, its right-continuous version $\alpha(x^\ast,z):=\inf_{\tilde{z}>z} \alpha_{F^{-1}(z)}(x^\ast)$ is a minimal penalty function.
    In particular, if $Z$ is properly preordered, $x^\ast \mapsto \alpha(x^\ast, z)$ is convex for every $z \in Z$.\footnote{Recall that the concept of a minimal penalty function includes the convexity Condition \ref{cond:min05} instead of \ref{cond:min05bis} if $Z$ is properly preordered.}
\end{proposition}
\begin{proof}
    \begin{enumerate}[label=\textit{Step \arabic*:}, fullwidth]
        \item By definition, $\alpha_{F^{-1}(\cdot)}(\cdot)$ maps from $X^\ast \times Z$ to $\mathcal{K}(\mathbb{R}, -\mathbb{R}_+)$.
            By Proposition \ref{prop:inverse}, $F^{-1}\colon Z \to \mathcal{P}(X,-\set{0})$ is increasing so that $\alpha_{F^{-1}(\cdot)}(\cdot)$ is increasing in the second argument.
            Hence, $\alpha$ fulfills condition \ref{cond:min01}. 
        \item By Proposition \ref{prop:uniquealpha}, $\alpha_{F^{-1}(z)}(\cdot)$ is convex and positively homogeneous for every $z$.
            The positive homogeneity of $\alpha_{F^{-1}(z)}(\cdot)$ for every $z \in Z$ implies
            \begin{equation*}
                \alpha(\lambda x^\ast,z)=\inf_{\tilde{z}>z}\alpha_{F^{-1}(\tilde{z})}(\lambda x^\ast)=\inf_{\tilde{z}>z}\lambda \alpha_{F^{-1}(\tilde{z})}(x^\ast)=\lambda \inf_{\tilde{z}>z}\alpha_{F^{-1}(\tilde{z})}(x^\ast)=\lambda \alpha(x^\ast,z)
            \end{equation*}
            showing that $\alpha(\cdot,z)$ is positively homogeneous for every $z$.

            Convexity implies that $\epi \alpha_{F^{-1}}(\cdot, z)$ is convex for every $z$.
            Hence, 
            \begin{equation*}
                \epi \alpha^-(\cdot,z)=\set{(x^\ast,z)\in X^\ast \times Z\colon  \alpha_{F^{-1}(\tilde{z})}(x^\ast)\eqslantless s\text{ for all }\tilde{z}< z}=\bigcap_{\tilde{z}<z} \epi \alpha_{F^{-1}(\tilde{z})}(\cdot)
            \end{equation*}
            which is convex.
            Thus, $\alpha^-$ is convex showing that $\alpha$ fulfills Condition \ref{cond:min05bis}.
            If $Z$ is additionally properly preordered,
            \begin{equation*}
                \epi \alpha(\cdot, z)=\Set{(x^\ast,s) \in X^\ast\times \R\colon \alpha_{F^{-1}(\tilde{z})}(x^\ast)\eqslantless s\text{ for some }\tilde{z} > z}=\bigcup_{\tilde{z}>z}\epi \alpha_{F^{-1}(\tilde{z})}(\cdot),
            \end{equation*}
            is also convex.
            Indeed, for every $z^1,z^2>z$ there exists $z^3>z$ with $z^1\geq z^3$ and $z^2\geq z^3$.
            Hence, the union of convex sets on the right-hand side is convex showing that $\alpha(\cdot, z)$ is convex for every $z$.
            Thus, if $Z$ is properly preordered, $\alpha$ fulfills Condition \ref{cond:min05}.
        \item By Proposition \ref{prop:uniquealpha}, $\alpha_{F^{-1}(z)}(\cdot)$ is upper level-closed, that is $\mathcal{L}_{\alpha_{F^{-1}(z)}(\cdot)}(s)$ is closed for every $s$.
            However
            \begin{equation*}
                \mathcal{L}_{\alpha^-(\cdot,z)}(s)=\Set{x^\ast \in X^\ast\colon \alpha_{F^{-1}(\tilde{z})}(x^\ast)\eqslantless s\text{ for all }\tilde{z}<z}=\bigcap_{\tilde{z}<z}\mathcal{L}_{\alpha_{F^{-1}(\tilde{z})}(\cdot)}(s)
            \end{equation*}
            showing that $\mathcal{L}_{\alpha^-(\cdot,z)}(s)$ is closed for every $s$.
            Therefore, together with \eqref{eq:LFRF02}, $\alpha$ fulfills Condition \ref{cond:min02}.
        \item By monotonicity, from $\alpha^-(x^\ast,z)=\emptyset$ follows $\alpha_{F^{-1}}(x^\ast, \tilde{z})=\emptyset$ for every $\tilde{z}<z$.
            However, $\alpha_{F^{-1}}(\cdot, \tilde{z})$ fulfills Condition \ref{cond:pen04} of Proposition \ref{prop:uniquealpha}.
            Hence, $\alpha_{F^{-1}}(\tilde{x}^\ast,\tilde{z})=\emptyset$ for every $\tilde{x}^\ast$ and $\tilde{z}<z$, and by definition $\alpha^-(\tilde{x}^\ast, z)=\emptyset$ for every $\tilde{x}^\ast$.
            Together with \eqref{eq:LFRF02}, $\alpha$ fulfills Condition \ref{cond:min03}.
    \end{enumerate}
Thus $\alpha$ is a minimal penalty function.
\end{proof}

\subsection{Proof of the Duality Theorem \ref{thm:robrep}}
\begin{proof}[Theorem \ref{thm:robrep}]
    \begin{enumerate}[label=\textit{Step \arabic*:},fullwidth]
        \item First, we show the existence of the dual representation \eqref{thm:main1} for some maximal risk function $R \in \mathcal{R}^{\max}$.
            Recall that $\alpha\colon X^\ast \times Z\to \mathcal{K}(\mathbb{R},-\mathbb{R}_+)$ is the right-continuous version of $z\mapsto \alpha_{F^{-1}(z)}(x^\ast)$ for every $x^\ast \in X^\ast$.
            Proposition \ref{prop:minpen} implies that $\alpha \in \mathcal{P}^{\min}$ so that by Proposition \ref{prop:minpen}, 
            \begin{equation*}
                (x^\ast,s)\longmapsto R(x^\ast,s)=\left[\alpha(x^\ast,\cdot)\right]^{-1}(s)=\Set{z\in Z \colon  s \eqslantless \alpha\left( x^\ast, z \right) }\in\mathcal{K}(Z,K)  
            \end{equation*}
            is a maximal risk function.
            Further, for fixed $s \in \mathbb{R}$ and $x^\ast \in X^\ast$, let us show that
            \begin{equation}
                \Set{z\in Z \colon  s \eqslantless \alpha^-\left( x^\ast, z  \right) }^\bullet= \Set{z\in Z \colon  s \eqslantless \alpha_{F^{-1}(z)}\left( x^\ast  \right) }^\bullet=R(x^\ast,s).
                \label{eq:eqAB}
            \end{equation}
            Indeed, let $A:=\Set{z\in Z \colon  s \eqslantless \alpha^-\left( x^\ast, z \right) }^\bullet$, $B:=\Set{z\in Z \colon  s\eqslantless \alpha_{F^{-1}(z)}\left( x^\ast  \right) }^\bullet$, and $C:=R(x^\ast,s)=\Set{z\in Z \colon  s\eqslantless \alpha\left( x^\ast, z \right)}^\bullet$.\footnote{Recall that $R(x^\ast,s)\in  \mathcal{K}(Z,K)$.}
            From $\alpha^-\eqslantless \alpha_{F^{-1}}\eqslantless \alpha$, it follows that $C\leqslant B\leqslant A$.
            Conversely, for every $z \in C$ and $k>0$, Relation \eqref{eq:LFRF01} yields $s \eqslantless  \alpha(x^\ast,z ) \eqslantless  \alpha^-(x^\ast, z+k)$.
            Hence, $z+k \in \set{z \in Z\colon s\eqslantless \alpha^-(x^\ast,z)}$ for every $k>0$, that is, $z\in A$.
            Thus, $A=B=C$ showing \eqref{eq:eqAB}.
            
            Since $F$ maps to $\mathcal{K}(Z,K)$, using Proposition \ref{lem:passage} and Relations \eqref{cond:pen01}, \eqref{eq:eqAB}, the same argumentation as in the proof of Proposition \ref{prop:existence} yields
            \begin{equation}
                \begin{split}
                 F(x)&=[F(x)]^\bullet =\Set{z \in  Z  \colon  x \in F^{-1}(z)}^\bullet=\Set{z\in Z \colon  \langle x^\ast,x\rangle \eqslantless \alpha_{F^{-1}(z)}\left( x^\ast  \right) \text{ for all }x^\ast\in X^\ast}^\bullet\\
                &=\left( \bigcap_{x^\ast \in X^\ast} \Set{z\in Z\colon  \langle x^\ast,x\rangle\eqslantless \alpha(x^\ast,z) } \right)^\bullet=\bigcap_{x^\ast \in X^\ast}\Set{z\in Z\colon  \langle x^\ast,x\rangle\eqslantless \alpha(x^\ast,z) }^\bullet\\
                &=\sup_{x^\ast \in X^\ast}R(x^\ast,\langle x^\ast,x\rangle)
                \end{split}
               \label{eq:eqBA}
            \end{equation}
            for $R \in \mathcal{R}^{\max}$.

        \item As for the uniqueness, suppose that
            \begin{equation*}
                F(x)=\sup_{x^\ast \in X^\ast}R_i(x^\ast,\langle x^\ast,x\rangle), \quad x \in X.
            \end{equation*} 
            for two maximal risk functions $R_i$, $i=1,2$.
            Denote by $\alpha_i$ their corresponding minimal penalty functions by means of Proposition \ref{prop:onetooneRalpha}.
            The same argumentation as above yields
            \begin{equation*}
                F(x)=\Set{z \in Z\colon  \langle x^\ast, x\rangle \eqslantless \alpha_i^-(x^\ast, z)\text{ for all }x^\ast \in X^\ast}^\bullet, \quad x \in X.
            \end{equation*}
            For a given $z \in Z$, it holds
            \begin{equation}\label{eq:blabla02}
                \begin{split}
                 H(z):=&\Set{x \in X\colon  F(x)\leqslant \tilde{z}\text{ for some }\tilde{z}<z}\\
                =&\Set{x \in X\colon  \Set{\hat{z}\in Z\colon \langle x^\ast,x\rangle \eqslantless \alpha_i^-\left(x^\ast,\hat{z}\right)\text{ for all }x^\ast \in X}^\bullet\leqslant \tilde{z}\text{ for some }\tilde{z}<z}\\
                =&\Set{x \in X\colon  \langle x^\ast,x\rangle\eqslantless \alpha_i^-(x^\ast,\tilde{z}+k)\text{ for all }x^\ast \in X^\ast\text{ all }k>0\text{ and some }\tilde{z}< z}\\
                =&\Set{x \in X\colon  \langle x^\ast,x\rangle\eqslantless \alpha_i^-(x^\ast,\tilde{z})\text{ for all }x^\ast \in X^\ast\text{ and some }\tilde{z}< z}\\
                =&\bigcup_{\tilde{z} < z}\Set{x \in X\colon   \langle x^\ast,x\rangle\eqslantless \alpha_i^-(x^\ast,\tilde{z})\text{ for all }x^\ast \in X^\ast}=:\bigcup_{\tilde{z}<z} A^{i,\tilde z},
                \end{split}
           \end{equation}
            where $A^{i,\tilde{z}}:=\set{x \in X\colon  \langle x^\ast,x\rangle\eqslantless \alpha_i^-(x^\ast,\tilde{z})\text{ for all }x^\ast \in X^\ast}$ is a closed convex set.
            Since $\alpha_i$ is a minimal penalty function, by means of Proposition \ref{prop:uniquealpha} we have
            \begin{equation}\label{eq:blabla01}
                \alpha_i^-(x^\ast,\tilde{z})=\sup_{x \in A^{i,z}} \langle x^\ast, x\rangle-\mathbb{R}_+, \quad x^\ast \in X^\ast.
            \end{equation}
            Relations \eqref{eq:blabla02} and \eqref{eq:blabla01} and the left-continuity of $\alpha_i^-(x^\ast,\cdot)$ yields
            \begin{equation*}
                \alpha_i^-(x^\ast,z)=\sup_{\tilde{z}< z }\alpha^-_i\left( x^\ast, \tilde{z} \right)=\sup_{\tilde{z}<z}\sup_{x \in A^{i,z}}\langle x^\ast,x\rangle -\mathbb{R}_+=\sup_{x \in H(z)}\langle x^\ast,x\rangle -\mathbb{R}_+
            \end{equation*}
            showing that $\alpha_1^-=\alpha_2^-$ and therefore $R_1=R_2$.

        \item Let us finally show the maximality assertion by considering two risk functions $R_i \in \mathcal{R}$ for which \eqref{thm:main1} holds and where $R_1 \in \mathcal{R}^{\max}$.
            Let $\alpha_i=R_i^{-1}$ be the inverse of $R_i$ for $i=1,2$, then $\alpha_1 \in \mathcal{P}^{\min}$ and $\alpha_2$ is right-continuous.
            Since 
            \begin{equation*}
                F(x)=\sup_{x^\ast \in X^\ast}R_i(x^\ast, \langle x^\ast,x\rangle), \quad x \in X\text{ and }i=1,2
            \end{equation*}
            it follows from this and \eqref{eq:C1}, that for every $z$
            \begin{align*}
                x \in F^{-1}(z)\quad \text{if, and only if,}\quad& F(x)\leqslant z\\
                \text{if, and only if,}\quad & R_i(x^\ast, \langle x^\ast, x\rangle)\leqslant z\quad \text{for all }x^\ast \in X^\ast\\
                \text{if, and only if,}\quad & \langle x^\ast, x\rangle\eqslantless \alpha_i(x^\ast , z)\quad \text{for all }x^\ast \in X^\ast.
            \end{align*}
            However, $\alpha_{F^{-1}(z)}$ is the smallest function for which the latter equivalence holds according to Proposition \ref{prop:uniquealpha}, therefore $\alpha_{F^{-1}}\eqslantless \alpha_i$ for $i=1,2$.
            Since both $\alpha_i$ are right-continuous and $\alpha_1=\alpha_{F^{-1}}^+$, it follows that $\alpha_1\eqslantless \alpha_2$.
            Hence, by means of Lemma \ref{lem:minmax} and Theorem \ref{thm:onetoone}, it follows that $R_2\leqslant R_1$ ending the proof.
    \end{enumerate}
\end{proof}

\section{Complete Duality in the Convex Case}\label{sec:04}

To stress the link with known results in set-valued convex analysis, in particular the Fenchel-Moreau type representation in \citep{hamel09}, we address the previous complete duality when the image space consists of closed monotone convex sets.
Throughout this section, $(Z,\leq)$ is a preordered vector space which is also a locally convex topological vector space.
We also assume that $K=\set{k \in Z\colon  k\geq 0}$ is such that $K\setminus (-K)\neq \emptyset$ and $K$ has a non-empty interior, that is, $\inte(K)\neq \emptyset$.
The topological dual $Z^\ast$ of $Z$ is equipped with the $\sigma(Z^\ast,Z)$-topology and we denote by 
\begin{equation*}
  K^\circ:=\set{k^\ast \in Z^\ast\colon  \langle k^\ast, k\rangle \geq 0\text{ for all }k \in K}  
\end{equation*}
the polar cone of $K$.
Further, $\inte(K)$ is a non-empty convex cone such that $\inte(K)\subseteq K\setminus(-K)$.
Indeed, if $k \in \inte(K)$ is such that $k,-k \in K$, then for a convex neighborhood $O$ of $0$ with $k+O\subseteq K$ and $-k+O\subseteq K$, it follows that $2O \subseteq K$.
Thus, $O$ being absorbing, $K=Z$ contradicting $K\setminus (-K)\neq \emptyset$.

We define the set of convex and monotone sets
\begin{align*}
    \mathcal{C}(Z,K)&:=\Set{A \in \mathcal{P}(Z,K)\colon  \co ( A )=A},\\
    \mathcal{C}(Z,-K)&:=\Set{A \in \mathcal{P}(Z,-K)\colon  \co( A)=A}
\end{align*}
as well as the sets of their closures
\begin{align*}
    \mathcal{G}(Z,K)&:=\Set{A \in \mathcal{C}(Z,K) \colon \cl (A)=A},\\
    \mathcal{G}(Z,-K)&:=\Set{A\in \mathcal{C}(Z,-K) \colon \cl (A)=A}
\end{align*}
where $\co(A)$ and $\cl(A)$ are the convex hull and closed hull of $A$ respectively.
Clearly, $\mathcal{G}(Z,K)\subseteq \mathcal{C}(Z,K)\subseteq \mathcal{P}(Z,K)$ and $\mathcal{G}(Z,-K)\subseteq \mathcal{C}(Z,-K)\subseteq \mathcal{P}(Z,-K)$.
\begin{proposition}\label{prop:convlat}
    The spaces $(\mathcal{C}(Z,K),\leqslant)$, $(\mathcal{C}(Z,-K),\eqslantless)$, $(\mathcal{G}(Z,K),\leqslant)$, $(\mathcal{G}(Z,-K),\eqslantless)$ are complete lattices with the lattice operations
    \begin{align*}
        \inf A_i&:=\co \left( \bigcup A_i \right) &\text{and}&& \sup A_i&=\bigcap A_i, &&(A_i)\subseteq \mathcal{C}(Z,K)\\
        \inf A_i&:= \bigcap A_i &\text{and}&& \sup A_i&=\co \left(\bigcup A_i\right), &&(A_i)\subseteq \mathcal{C}(Z,-K)\\
        \inf A_i&:=\cl\,\co\left( \bigcup A_i \right) &\text{and}&& \sup A_i&=\bigcap A_i, &&(A_i)\subseteq \mathcal{G}(Z,K)\\
        \inf A_i&:=\bigcap A_i  &\text{and}&& \sup A_i&=\cl\, \co\left(\bigcup A_i\right), &&(A_i)\subseteq \mathcal{G}(Z,-K).
    \end{align*}
\end{proposition}
In each case, the proof of the complete lattice property is straightforward, see \citep{hamel09}.

\begin{remark}
    In Subsection \ref{subsec:bullet}, the monotone closure ${}^\bullet$ of sets is defined with respect to $\hat{K}=K\setminus(-K)$.
    The same operation can, however, be defined with respect to any non-empty convex cone included in $\hat{K}$, in particular for $\inte(K)$.
    From now on we keep the notation
    \begin{equation*}
        A^\bullet=\Set{z \in Z\colon  z+\inte(K)\subseteq A}.
    \end{equation*}
    We further write $z <\tilde{z}$ whenever $\tilde{z}-z \in \inte(K)$.\footnote{Note that this strict order is still compatible with the vector structure since $\inte(K)$ is a convex cone, however it is different from the asymmetric part of $\leq$.}
\end{remark}
The following assertion shows that the ${}^\bullet$ operator for monotone sets coincides with the closure.
\begin{proposition}\label{prop:closed}
    If $A \in \mathcal{P}(Z,K)$, then $\cl(A)=A^\bullet$.
\end{proposition}
\begin{proof}
    First, Proposition \ref{lem:passage} implies that $[\cl(A)]^\bullet \leqslant \cl(A)\leqslant A$ and Proposition \ref{prop:convlat} yields $\cl(A)=[\cl(A)]^\bullet$.
    Hence, by Proposition \ref{lem:passage} it holds that $\cl(A)\leqslant A^\bullet$.
    Second, we show that for $z \in \cl(A)$, $z+k \in A$ for every $k>0$, that is, $k\in \inte(K)$.
    Let $O$ be neighborhood of $0$ such that $k+O\subseteq K$.
    Let further $V$ be another neighborhood of $0$ such that $V+V\subseteq O$.
    Fix $\tilde{z}\in A$ such that $z-\tilde{z} \in V$.
    It follows that
    \begin{equation*}
        z+k+V=\tilde{z}+k+z-\tilde{z}+V\subseteq \tilde{z}+k+V+V\subseteq \tilde{z}+k+O\subseteq \tilde{z}+k+K\subseteq A
    \end{equation*}
    which implies that $z+k \in A$ for all $k>0$.
    This shows that $A^\bullet \leqslant \cl(A)$.
\end{proof}
Let $(X,\leq)$ be another preordered locally convex topological vector space with $C:=\set{c \in X\colon  c\geq 0}$ and $\inte(C)\neq \emptyset$.
For increasing functions $F\colon  X \to \mathcal{G}(Z,K)$ and $G\colon Z \to \mathcal{G}(X,-C)$ the left- and right-continuous version of $F$ and $G$ are given by
\begin{align*}
    F^-(x)&=\sup_{\tilde{x}<x}F(\tilde{x})=\bigcap_{c \in \inte(C)} F(x+c)&&\text{and}&&&G^-(z)&=\sup_{\tilde{z}<z}G(\tilde{z})=\bigcap_{k \in \inte(K)} G(z+k).
\end{align*}

\begin{proposition}\label{prop:onetoone}
    The following statements hold
    \begin{enumerate}[label=\textit{(\roman*)}]
        \item $F\colon X\to \mathcal{C}(Z,K)$ is increasing and quasiconvex if, and only if, $F^{-1}\colon Z \to \mathcal{C}(X,-C)$ is increasing and quasiconcave;
        \item $F\colon X\to \mathcal{G}(Z,K)$ is increasing, lower level-closed, left-continuous and quasiconvex if, and only if, $F^{-1}\colon Z \to \mathcal{G}(X,-C)$ is increasing, upper level-closed, right-continuous and quasiconcave.
            In this case, the Relations \eqref{eq:C1} and \eqref{eq:C2} hold.
    \end{enumerate}
\end{proposition}
\begin{proof}
    \begin{enumerate}[label=\textit{(\roman*)},fullwidth]
        \item Let $F\colon X\to \mathcal{C}(Z,K)$ be an increasing and quasiconvex function.
            Proposition \ref{prop:inverse} implies that $F^{-1}\colon Z \to \mathcal{P}(X,-C)$ is increasing and Proposition \ref{prop:quasiconvex} implies that $F^{-1}$ takes values in $\mathcal{C}(X,-C)$.
            Let now $\lambda \in (0,1)$ and $z^1,z^2 \in Z$.
            If $F(x)\leqslant z^1$ and $F(x)\leqslant z^2$, since $F(x)$ is convex, it follows that $F(x)=\lambda F(x)+(1-\lambda)F(x)\leqslant \lambda z^1+(1-\lambda)z^2$.
            Hence $F^{-1}(z^1)\eqslantless x$ and $F^{-1}(z^2)\eqslantless x$ implies $F^{-1}(\lambda z^1+(1-\lambda)z^2)\eqslantless x$, that is, $\inf\set{F^{-1}(z^1),F^{-1}(z^2)}\eqslantless F^{-1}(\lambda z^1+(1-\lambda)z^2)$.
            Thus $F$ is quasiconcave.
            The converse statement is analogous.
        \item Let $F\colon X\to \mathcal{G}(Z,K)$ be an increasing, lower level-closed, left-continuous and quasiconvex function.
            It follows from the previous step that $F^{-1}\colon Z \to \mathcal{C}(X,-C)$ is an increasing, quasiconvex function.
            However, $F$ being lower level-closed, it follows that $F^{-1}$ takes values in $\mathcal{G}(X,-C)$.
            Further, 
            \begin{equation*}
                \mathcal{U}_{F^{-1}}(x)=\set{z \in Z\colon  x\eqslantless F^{-1}(z)}=\set{z \in Z\colon  F(x)\leqslant z}=F(x)   
            \end{equation*}
            which is closed.
            Hence $F^{-1}$ is upper level-closed.
            Finally, since $\mathcal{G}(Z,K)\subseteq \mathcal{K}(Z,K)$ as well as $\mathcal{G}(X,-C)\subseteq \mathcal{K}(X,-C)$, the statement follows from Theorem \ref{thm:onetoone}.
    \end{enumerate}
\end{proof}
We can formulate the complete duality result for lower level-closed and quasiconvex set-valued functions taking values in $\mathcal{G}(Z,K)$.
A function $R\colon X^\ast \times \mathbb{R}\to \mathcal{G}(Z,K)$ is called a \emph{maximal risk function} if it fulfills \ref{cond:max01}, \ref{cond:max05}, \ref{cond:max02}, and \ref{cond:max03} of the previous section.
\begin{theorem}\label{thm:robrepint}
    Any lower level-closed and quasiconvex function $F\colon X \to \mathcal{G}(Z,K)$ admits the dual representation
    \begin{equation}
        F\left(x\right)= \sup_{x^\ast \in  X^\ast } R\left(x^\ast,\langle x^\ast,x\rangle\right),
        \label{thm:main2}
    \end{equation}
    for a unique maximal risk function $R\colon X^\ast \times \mathbb{R}\to \mathcal{G}(Z,K)$.

    Furthermore, if \eqref{thm:main2} holds for another $\tilde{R}\colon X^\ast \times \mathbb{R}\to \mathcal{G}(Z,K)$ satisfying Condition \ref{cond:max01}, then $\tilde{R}\leqslant R$.
\end{theorem}
\begin{proof}
    \begin{enumerate}[label=\textit{Step \arabic*:}, fullwidth]
        \item Let us show that $R\colon X^\ast \times \mathbb{R}\to \mathcal{G}(Z,K)$ is a maximal risk function if, and only if, $\alpha=R^{-1}\colon  X^\ast \times Z\to \mathcal{G}(\mathbb{R},-\mathbb{R}_+)$ is a minimal penalty function, that is, it satisfies \ref{cond:min05}, \ref{cond:min02}, \ref{cond:min03} and 
            \begin{enumerate}[label=(\alph*')]
                \item \label{cond:min01bis} $z\mapsto \alpha(x^\ast,z)$ is increasing, upper level-closed, right-continuous and quasiconcave.
            \end{enumerate}
            instead of \ref{cond:min01}.
            According to the proof of Proposition \ref{prop:onetooneRalpha}, we just have to show that \ref{cond:max01} for $R$ is equivalent to \ref{cond:min01bis} for $\alpha$.
            Proposition \ref{prop:onetoone} shows that $\alpha$ fulfills \ref{cond:min01bis} if, and only if, $R\colon X^\ast \times \mathbb{R}\to \mathcal{G}(Z,K)$ fulfills 
            \begin{enumerate}[label=(\roman*')]
                \item \label{cond:max01bis} $s\mapsto R(x^\ast,s)$ is increasing, lower level-closed, left-continuous and quasiconvex.
            \end{enumerate}
            However, since $\mathbb{R}$ is totally preordered, \ref{cond:max01bis} is equivalent \ref{cond:max01}.
            Indeed, \ref{cond:max01bis} implies \ref{cond:max01} is immediate.
            Conversely, for $s_1,s_2 \in \mathbb{R}$, and $\lambda \in (0,1)$, without loss of generality, $s_1\leq s_2$, it follows from $R$ being increasing, that $R(x^\ast, \lambda s_1+(1-\lambda)s_2)\leqslant R(x^\ast, s_2)=\sup\set{ R(x^\ast, s_1), R(x^\ast,s_2)}$ showing the quasiconcavity.
            Further, from the left-continuity it follows that $\set{s \in \mathbb{R}\colon  R(x^\ast, s)\leqslant z}=]-\infty, s_0] \in \mathcal{G}(\mathbb{R},-\mathbb{R}_+)$ for some $s_0$ showing the lower semi-continuity.

        \item Since $\inte(K)\neq \emptyset$, we are properly preordered for $<$.
            Indeed, for $k_1,k_2 \in \inte(K)$, it follows that $k_1-\inte(K)$ and $k_2-\inte(K)$ are two neighborhoods of $0$.
            Hence, $(k_1-\inte(K))\cap (k_2-\inte(K)) \cap \inte(K)\neq \emptyset$ showing the existence of $k_3 \in \inte(K)$ such that $k_3\leq k_1,k_2$.

        \item Let us show that the right-continuous version $\alpha$ of $\alpha_{F^{-1}}$ for $F\colon X \to \mathcal{G}(X,K)$ lower level-closed and quasiconvex fulfills Conditions \ref{cond:min01bis}, \ref{cond:min05}, \ref{cond:min02} and \ref{cond:min03}.
            As for \ref{cond:min05}, \ref{cond:min02} and \ref{cond:min03}, it follows by the same argumentation as Proposition \ref{prop:minpen}, since we are properly preordered.
            Let us show that $z\mapsto \alpha(x^\ast, z)$ fulfills \ref{cond:min01bis}.
            Being increasing and right-continuous follows from Theorem \ref{thm:onetoone}.
            According to Proposition \ref{prop:onetoone}, $F^{-1}\colon  Z \to \mathcal{G}(X,-\set{0})$ is quasiconcave, hence
            \begin{equation*}
                \mathcal{U}_{\alpha(x^\ast,\cdot)}(s)=\Set{z \in Z\colon  s\eqslantless \inf_{\tilde{z}>z}\alpha_{F^{-1}(\tilde{z})}(x^\ast)}=\Set{z \in Z\colon  s\leq \sup_{x \eqslantless F^{-1}(\tilde{z})}\langle x^\ast, x\rangle \text{ for all }\tilde{z}>z}\\
            \end{equation*}
            is convex since $x \eqslantless \inf \set{F^{-1}(z^1),F^{-1}(z^2)}\eqslantless F^{-1}(\lambda z^1+(1-\lambda)z^2)$.
            We are left to show that $\mathcal{U}_{\alpha(x^\ast,\cdot)}(s)$ is closed for all $s$.
            Let $z \in Z$ such that $z+k\in\mathcal{U}_{\alpha(x^\ast,\cdot)}(s)$ for all $k>0$.
            It follows that $s\leq \sup_{x \in F^{-1}(z+k+\hat{k})}\langle x^\ast, x\rangle$ for every $k>0$ and $\hat{k}>0$.
            Hence $\mathcal{U}_{\alpha(x^\ast,\cdot)}(s)=[\mathcal{U}_{\alpha(x^\ast,\cdot)}(s)]^\bullet$, and by Proposition \ref{prop:closed} it is therefore closed.

        \item Finally, existence, uniqueness and minimality follows by the same argumentation as in the proof of Theorem \ref{thm:robrep} since the closure and the ${}^\bullet$ operation coincide.
    \end{enumerate}
\end{proof}

We now address the unique characterization of the Fenchel-Moreau type representation in \citep{hamel09} in the case where the interior of $K=\set{k \in Z\colon  k\geq 0}$ is non-empty.
For the sake of completeness, we also sketch the proof of the existence.\footnote{As in \citep{hamel09}, the subsequent proof of the existence does not require $\inte(K)\neq \emptyset$.}
To simplify notations, we write $A+B$ for $\cl \co (A+B)$ where $A,B \in \mathcal{G}(Z,K)$.
For every $z^\ast \in K^\circ$, we define the function $S\colon K^\circ\times \R \to \mathcal{G}(Z,K)$ as follows
\begin{equation*}
    S(z^\ast,s)=\Set{z \in Z\colon  s\leq \langle z^\ast,z\rangle}.
\end{equation*}
Note that $\lambda S(z^\ast, s)=S(z^\ast,\lambda s)$ for every $\lambda >0$.
Likewise, it also holds $\lambda S(z^\ast, s_1)+(1-\lambda)S(z^\ast, s_2)=S(z^\ast, \lambda s_1+(1-\lambda)s_2)$ for every $\lambda \in (0,1)$, that is, $S(z^\ast,\cdot)$ is affine.
Finally, for $F\colon X \to \mathcal{G}(Z,K)$ we define $-G\colon X^\ast \times K^\circ \to \mathcal{G}(Z,K)$ as
\begin{equation}
    -G(x^\ast,z^\ast)=\inf_{x \in X} \Set{F(x)+S(z^\ast,-\langle x^\ast, x\rangle)}.
    \label{}
\end{equation}
This functional is the set-valued Fenchel-Moreau conjugate introduced in \citet{hamel09} and can be seen as an analogue to the negative Fenchel-Moreau convex conjugate in the scalar case.
However, unlike the Fenchel-Moreau conjugate, it is not an automorphism.
\begin{theorem}\label{thm:robrepconv}
    Let $F\colon X \to \mathcal{G}(Z,K)$ be a lower level-closed convex function which is proper, that is $F(x)\neq \emptyset$ for some $x\in X$ and $F(x)\neq Z$ for all $x \in X$.
    Then, it holds
    \begin{equation}\label{eq:FM}
        F(x)=\sup_{x^\ast \in X^\ast}\sup_{z^\ast \in K^\circ}\Set{ -G(x^\ast, z^\ast)+S\left(z^\ast,\langle x^\ast, x\rangle\right) }.
    \end{equation}
    Furthermore, denoting by $R$ the corresponding unique maximal risk function according to Theorem \ref{thm:robrepint}, it holds
    \begin{equation*}
        R(x^\ast,s)=\sup_{z^\ast \in K^\circ}\Set{ -G(x^\ast, z^\ast)+S\left(z^\ast,s\right)}
    \end{equation*}
    for every $x^\ast \in X^\ast$ such that $-G(x^\ast, z^\ast)\neq \emptyset$ for some $z^\ast \in K^\circ \setminus \set{0}$.
\end{theorem}
\begin{proof}
    Define $\chi(x,z)=0$ if $z \in F(x)$ and $\infty$ otherwise.
    Since $\set{(x,z)\in X\times Z\colon  \chi(x,z)\leq 0}=\epi F$, it follows that $\chi$ is jointly lower semi-continuous and jointly convex.
    Furthermore, since $F$ is proper it follows that $\chi$ is jointly proper.
    Finally, $\chi$ is decreasing in $z$.
    Hence
    \begin{align*}
        F(x)=\Set{z \in Z\colon  \chi(x,z)\leq 0}&=\Set{z \in Z\colon  \sup_{x^\ast \in X^\ast}\sup_{z^\ast \in K^\circ}\Set{\langle x^\ast, x\rangle +\langle - z^\ast, z\rangle -\chi^\ast\left( x^\ast, -z^\ast \right)}\leq 0}\\
        &=\sup_{x^\ast \in X^\ast}\sup_{z^\ast \in K^\circ}\Set{z \in Z\colon \langle x^\ast, x\rangle- \chi^\ast(x^\ast, - z^\ast) \leq \langle z^\ast,z\rangle }\\
        &=\sup_{x^\ast \in X^\ast}\sup_{z^\ast \in K^\circ}\Set{\Set{z \in Z\colon  \langle z^\ast,z\rangle \geq  -\chi^\ast(x^\ast,- z^\ast)}+S\left(z^\ast, \langle x^\ast, x\rangle \right)}.
    \end{align*}
    Indeed,
    \begin{equation}\label{eq:globu}
        \Set{\Set{z \in Z\colon  \langle z^\ast,z\rangle \geq  -\chi^\ast(x^\ast,- z^\ast)}+S\left(z^\ast, \langle x^\ast, x\rangle \right)}\geqslant \Set{z \in Z\colon \langle x^\ast, x\rangle - \chi^\ast(x^\ast,-z^\ast)\leq \langle z^\ast ,z \rangle }.
    \end{equation}
    If $z^\ast=0$ the equality trivially holds.
    For $z^\ast \neq 0$, there exists $z^1\in S(z^\ast,\langle x^\ast,x\rangle )$ such that $\langle z^\ast, z^1\rangle =\langle x^\ast,x\rangle$.
    Hence, for every $z$ in the right hand side of \eqref{eq:globu}, it holds 
    \begin{equation*}
        z^2=z-z^1 \in \Set{z \in Z\colon  \langle z^\ast, z\rangle \geq -\chi^\ast(x^\ast,-z^\ast)}
    \end{equation*}
    which shows the reverse inequality.
    
    Further, $\Set{z \in Z\colon  \langle z^\ast,z\rangle \geq -\chi^\ast(x^\ast, -z^\ast)} \in \mathcal{G}(Z, K)$ for every $x^\ast \in X^\ast$ and $z^\ast \in K^\circ$ and it holds
    \begin{equation*}
        \Set{z \in Z\colon  \langle z^\ast,z\rangle \geq -\chi^\ast(x^\ast, -z^\ast)}=\inf_{x \in X}\Set{F(x)+S(z^\ast, -\langle x^\ast, x\rangle )}.
    \end{equation*}
    Indeed, on the first hand, for $z_1 \in S(z^\ast, -\langle x^\ast, x\rangle )$ and $z_2 \in  F(x)$, it follows that $\langle z^\ast, z_1+z_2\rangle \geq -\langle x^\ast,x\rangle +\langle z^\ast, z_2\rangle+\chi(x,z_2)\geq -\chi^\ast(x^\ast, -z^\ast)$ so that
    \begin{equation}\label{eq:globi}
        \inf_{x \in X}\Set{F(x)+S(z^\ast, -\langle x^\ast, x\rangle )}\geqslant \set{z \in Z\colon  \langle z^\ast,z\rangle \geq -\chi^\ast(x^\ast, -z^\ast)}.
    \end{equation}
    On the other hand, if $z^\ast=0$ then equality is direct.
    If $z^\ast \neq 0$, let $z$ be such that $\langle z^\ast, z\rangle > -\chi^\ast(x^\ast, -z^\ast)$.
    We can find $x \in X$ and $z_2 \in F(x)$ such that $\langle z^\ast, z\rangle \geq -\langle x^\ast, x\rangle +\langle z^\ast, z_2\rangle$.
    Hence $z^1=z-z^2 \in S(z^\ast, -\langle x^\ast, x\rangle)$.
    From $\inf_{x \in X }\set{F(x)+S(z^\ast, -\langle x^\ast, x\rangle)}$ being closed the reverse inequality in \eqref{eq:globi} follows.
    
    As for the uniqueness, let $R$ be the unique maximal risk function corresponding to $F$ according to Theorem \ref{thm:robrepint} and define $\tilde{R}(x^\ast,s):= \sup_{z^\ast \in K^\circ}\Set{-G(x^\ast,z^\ast)+S(z^\ast,s)} \in \mathcal{G}(Z,K)$.
    The previous argumentation shows in particular that
    \begin{align*}
        \tilde{R}(x^\ast, s)&=\sup_{z^\ast \in K^\circ}\Set{z \in Z\colon  \langle z^\ast,z\rangle \geq s-\chi^\ast(x^\ast,-z^\ast)}\\
        &=\Set{z \in Z\colon  s\leq \inf_{z^\ast \in K^\circ} \Set{\langle z^\ast, z\rangle +\chi^\ast(x^\ast, -z^\ast}}.
    \end{align*}
    Inspection shows that $\tilde{R}(x^\ast, \cdot)$ is left-continuous and increasing.
    Fix $x^\ast \in X^\ast$ such that $-G(x^\ast, \tilde{z}^\ast)\neq \emptyset$ for some $\tilde{z}^\ast \in K^\circ \setminus \set{0}$ which by definition of $G$ means that $\chi^\ast(x^\ast, -\tilde{z}^\ast)<\infty$.
    Now, showing that $\tilde{R}(x^\ast, \cdot)=R(x^\ast, \cdot)$ is equivalent to the fact that $\tilde{\alpha}(x^\ast, \cdot)=\alpha(x^\ast, \cdot)$ where $\tilde{\alpha}$ and $\alpha$ are the inverse in the second argument of $\tilde{R}$ and $R$ respectively.
    According to Proposition \ref{prop:minpen}, $\alpha(x^\ast, z)=\inf_{\tilde{z}>z} \sup_{x \in F^{-1}(\tilde{z})}\langle x^\ast, x\rangle$.
    However, the definition of the inverse and the fact that $\tilde{z}\mapsto -\chi(x,\tilde{z})$ is increasing yield
    \begin{equation}\label{eq:uoapala}
        \begin{split}
            \tilde{\alpha}(x^\ast,z)=&\Set{s \in \mathbb{R}\colon  z \in R(x^\ast,s)}\\
            =&\Set{s \in \mathbb{R}\colon  \inf_{z^\ast \in K^\circ}\sup_{\tilde{z}\in Z}\sup_{x \in X} \Set{\langle x^\ast, x\rangle +\langle z^\ast, z-\tilde{z}\rangle -\chi(x,\tilde{z})}\geq s}\\
            \eqslantgtr & \Set{s \in \mathbb{R}\colon  \sup_{x \in X}\sup_{\tilde{z}\in Z}\inf_{z^\ast \in K^\circ} \Set{\langle x^\ast, x\rangle+\langle z^\ast, z-\tilde{z}\rangle -\chi(x,\tilde{z})}\geq s}\\
            =&\Set{s \in \mathbb{R}\colon  \sup_{x \in X}\sup_{\tilde{z}\leq z}\Set{\langle x^\ast, x\rangle -\chi(x,\tilde{z})}\geq s}=\Set{s \in \mathbb{R}\colon  \sup_{x \in X} \Set{\langle x^\ast, x\rangle -\chi(x,z)}\geq s}\\
            =&\Set{s \in \mathbb{R}\colon  \sup_{x \in F^{-1}(z)}\langle x^\ast, x\rangle \geq s}=\alpha_{F^{-1}(z)}(x^\ast).
        \end{split}
    \end{equation}
    Hence, $\tilde{\alpha}$ and $\alpha$ being right-continuous, it follows that $\tilde{\alpha}(x^\ast, \cdot)\eqslantgtr \alpha(x^\ast, \cdot)$.
    By means of Relation \eqref{eq:LFRF02}, it is enough to show that $\alpha_{F^{-1}(z_1)}(x^\ast)\eqslantgtr \alpha(x^\ast, z_2)$ as soon as $z_1>z_2$.
    We therefore fix $z_1> z_2$.
    For every $z^\ast\in K^\circ$, since $z \mapsto \langle x^\ast, x\rangle -\chi(x,z)$ is increasing, it holds
    \begin{align*}
        \langle z^\ast,z_1-z_2\rangle +\sup_{x \in F^{-1}(z_1)}\langle x^\ast, x\rangle &= \langle z^\ast, z_1-z_2\rangle +\sup_{x \in X}\Set{\langle x^\ast, x\rangle -\chi(x,z_1)}\\
        &= \langle z^\ast, z_1-z_2\rangle +\sup_{x \in X}\sup_{z \leq z_1}\Set{\langle x^\ast, x\rangle -\chi(x,z)}\\
        &\geq \sup_{x \in X}\sup_{z_2-(z_1-z_2)\leq z\leq z_1}\Set{ \langle x^\ast, x\rangle +\langle z^\ast , z_2-z\rangle -\chi(x,z)}.
    \end{align*}
    But for $\tilde{z}^\ast \in K^\circ\setminus \set{0}$ we have $\chi^\ast(x^\ast, -\tilde{z}^\ast)<\infty$, and so it holds
    \begin{align*}
        &\inf_{\lambda>0}\Set{\langle \lambda \tilde{z}^\ast, z_2\rangle +\chi^\ast(x^\ast, - \lambda \tilde{z}^\ast)} = \inf_{\lambda>0}\sup_{x \in X}\sup_{z \in Z}\Set{\langle x^\ast, x\rangle +\langle \lambda \tilde{z}^\ast, z_2-z\rangle-\chi(x,z) }\\
        &\qquad \qquad \leq \sup_{x \in F^{-1}(z_1)}\langle x^\ast, x\rangle +\inf_{\lambda >0} \sup_{x \in X}\sup_{z\leq z_2-(z_1-z_2)}\Set{\langle x^\ast, x\rangle +\langle \lambda \tilde{z}^\ast, z_2-z\rangle-\chi(x,z) }\\
        &\qquad \qquad\qquad \qquad +\inf_{\lambda>0}\sup_{x \in X}\sup_{z_1\leq z }\Set{\langle x^\ast, x\rangle +\langle \lambda \tilde{z}^\ast, z_2-z\rangle-\chi(x,z) }\\
        &\qquad \qquad \leq \sup_{x \in F^{-1}(z_1)}\langle x^\ast, x\rangle +\inf_{\lambda >0} \lambda \Set{\langle\tilde{z}^\ast, z_2\rangle +\sup_{x \in X}\sup_{z\in Z}\Set{\langle x^\ast, x\rangle +\langle  \tilde{z}^\ast, -z\rangle-\chi(x,z) }}\\
        &\qquad \qquad\qquad \qquad     +\inf_{\lambda>0}\lambda \Set{\langle z^\ast, z_2\rangle +\sup_{x \in X}\sup_{ z\in Z }\Set{\langle x^\ast, x\rangle +\langle \tilde{z}^\ast, -z\rangle-\chi(x,z) }}\\
        &\qquad \qquad     \leq \sup_{x \in F^{-1}(z_1)}\langle x^\ast, x\rangle +2\inf_{\lambda >0} \lambda \Set{\langle \tilde{z}^\ast, z_2\rangle +\chi^\ast(x^\ast, -z^\ast)}=\sup_{x \in F^{-1}(z_1)}\langle x^\ast, x\rangle.
    \end{align*}
    Hence, $\alpha_{F^{-1}}(z_1)\eqslantgtr \tilde{\alpha}(x^\ast, z_2)$, showing that $\tilde{\alpha}(x^\ast,\cdot)=\alpha(x^\ast, \cdot)$.
    Thus $\tilde{R}(x^\ast, \cdot)=R(x^\ast,\cdot)$ for every $x^\ast \in X^\ast$ such that $G(x^\ast, z^\ast)\neq Z$ for some $z^\ast \in K^\circ \setminus{0}$.
\end{proof}
\begin{remark}
    Even if $G$ characterizes the maximal risk function on the set of those $x^\ast \in X^\ast$ for which $G(x^\ast,z^\ast)\neq Z$ for some $z^\ast \neq 0$, it yields in general a strictly smaller risk function on---for the representation irrelevant---set of those $x^\ast \in X^\ast$ satisfying $G(x^\ast, z^\ast)=Z$ for all $z^\ast \in K^\circ \setminus \set{0}$.
    This is also true in the scalar case as the following example shows.
    Let $F(x)=f(x)+\mathbb{R}_+$ where $f(x)=x\vee 0$, with $x \in \mathbb{R}$.
    Straightforward computation shows that $\chi(-1,-z^\ast)=\infty$ for all $z^\ast \in \mathbb{R}_+$.
    Hence,
    \begin{equation*}
        \tilde{R}(-1,s)=\sup_{z^\ast \in \mathbb{R}_+}\Set{T\left( z^\ast, s \right)-G(-1,z^\ast)}=\mathbb{R}.
    \end{equation*}
    However, the maximal risk function $R(-1,s)$ is the inverse of the right continuous version of $z \mapsto \sup_{x \in F^{-1}(z)} - x-\mathbb{R}_+$.
    Hence, $\sup_{x \in F^{-1}(z)} - x=\sup_{x \in \emptyset} -x=-\infty$ for every $z<0$, whereas $\sup_{x \in F^{-1}(z)} - x=+\infty$ for all $z\geq 0$.
    This shows 
    \begin{equation*}
        R(-1,s)=\mathbb{R}_+
    \end{equation*}
    which is strictly greater than $\tilde{R}(-1,s)=\mathbb{R}$ for all $s \in \mathbb{R}$.
\end{remark}

\bibliographystyle{abbrvnat}
\bibliography{biblio}

\begin{thebibliography}{15}
\providecommand{\natexlab}[1]{#1}
\providecommand{\url}[1]{\texttt{#1}}
\expandafter\ifx\csname urlstyle\endcsname\relax
  \providecommand{\doi}[1]{doi: #1}\else
  \providecommand{\doi}{doi: \begingroup \urlstyle{rm}\Url}\fi

\bibitem[Benoist and Popovici(2003)]{BenoistPopovici03MMOR}
J.~Benoist and N.~Popovici.
\newblock Characterizations of convex and quasiconvex set-valued maps.
\newblock \emph{Mathematical Methods of Operations Research}, 57\penalty0
  (3):\penalty0 427--435, 2003.

\bibitem[Benoist et~al.(2003)Benoist, Borwein, and
  Popovici]{BenoistBorweinPopovici03PAMS}
J.~Benoist, J.~Borwein, and N.~Popovici.
\newblock A characterization of quasiconvex vector-valued functions.
\newblock \emph{Proceedings of the American Mathematical Society}, 131\penalty0
  (4):\penalty0 1109--1113, 2003.

\bibitem[Cerreia-Vioglio et~al.(2011{\natexlab{a}})Cerreia-Vioglio, Maccheroni,
  Marinacci, and Montrucchio]{marinacci03}
S.~Cerreia-Vioglio, F.~Maccheroni, M.~Marinacci, and L.~Montrucchio.
\newblock Uncertainty averse preferences.
\newblock \emph{Journal of Economic Theory}, 146\penalty0 (4):\penalty0
  1275--1330, 2011{\natexlab{a}}.

\bibitem[Cerreia-Vioglio et~al.(2011{\natexlab{b}})Cerreia-Vioglio, Maccheroni,
  Marinacci, and Montrucchio]{marinacci04}
S.~Cerreia-Vioglio, F.~Maccheroni, M.~Marinacci, and L.~Montrucchio.
\newblock Complete monotone quasiconcave duality.
\newblock \emph{Mathematics of Operations Research}, 36\penalty0 (2):\penalty0
  321--339, 2011{\natexlab{b}}.

\bibitem[Drapeau and Kupper(2013)]{drapeau01}
S.~Drapeau and M.~Kupper.
\newblock Risk preferences and their robust representation.
\newblock \emph{Mathematics of Operations Research}, 28\penalty0 (1):\penalty0
  28--62, 2013.

\bibitem[Hamel(2009)]{hamel09}
A.~Hamel.
\newblock A duality theory for set-valued functions {I}: Fenchel conjugation
  theory.
\newblock \emph{Set-Valued and Variational Analysis}, 17:\penalty0 153--182,
  2009.

\bibitem[Hamel and Heyde(2010)]{HamelHeyde10SIFIN}
A.~H. Hamel and F.~Heyde.
\newblock Duality for set-valued measures of risk.
\newblock \emph{SIAM Journal of Financial Mathematics}, 1\penalty0
  (1):\penalty0 66--95, 2010.

\bibitem[Hamel and L\"ohne(2014)]{HamelLoehne13JOTA-OF}
A.~H. Hamel and A.~L\"ohne.
\newblock Lagrange duality in set optimization.
\newblock \emph{Journal of Optimization Theory and Applications}, 161\penalty0
  (2):\penalty0 368--397, 2014.

\bibitem[Hamel et~al.(2011)Hamel, Heyde, and Rudloff]{HamelHeydeRudloff11MAFE}
A.~H. Hamel, F.~Heyde, and B.~Rudloff.
\newblock Set-valued risk measures for conical market models.
\newblock \emph{Mathematics and Financial Economics}, 5\penalty0 (1):\penalty0
  1--28, 2011.

\bibitem[Jouini et~al.(2004)Jouini, Meddeb, and Touzi]{JouiniMeddebTouzi04FS}
E.~Jouini, M.~Meddeb, and N.~Touzi.
\newblock Vector-valued coherent risk measures.
\newblock \emph{Finance \& Stochastics}, 8\penalty0 (4):\penalty0 531--552,
  2004.

\bibitem[L{\"o}hne(2011)]{Loehne11Book}
A.~L{\"o}hne.
\newblock \emph{Vector optimization with infimum and supremum}.
\newblock Springer, 2011.

\bibitem[Luc(1987)]{DinhTheLuc87JMAA}
D.~T. Luc.
\newblock Connectedness of the efficient point sets in quasiconcave vector
  maximization.
\newblock \emph{Journal of Mathematical Analysis and Applications},
  122\penalty0 (2):\penalty0 346--354, 1987.

\bibitem[Penot and Volle(1990{\natexlab{a}})]{penot01}
J.-P. Penot and M.~Volle.
\newblock On quasi-convex duality.
\newblock \emph{Mathematics of Operations Research}, 15:\penalty0 597--625,
  1990{\natexlab{a}}.

\bibitem[Penot and Volle(1990{\natexlab{b}})]{penot02}
J.-P. Penot and M.~Volle.
\newblock Inversion of real-valued functions and applications.
\newblock \emph{Mathematical Methods of Operations Research}, 34\penalty0
  (2):\penalty0 117--141, 1990{\natexlab{b}}.

\bibitem[Shephard(1970)]{Shephard70Book}
R.~W. Shephard.
\newblock \emph{Theory of cost and production functions}.
\newblock Princeton University Press, Princeton, 1970.

\end{thebibliography}
\end{document}